\documentclass[review,hidelinks,onefignum,onetabnum]{siamart220329}
\nolinenumbers

\usepackage{amssymb,latexsym} %
\usepackage{mathrsfs}%
\usepackage{mathtools} 
\usepackage{bm}%
\usepackage{subfigure}
\usepackage{float}
\usepackage{cases} 
\usepackage{textcomp} %
\usepackage{hyperref}
\usepackage{booktabs}
\usepackage{color}
\usepackage{multirow}%
\usepackage{longtable}
\usepackage{cleveref}
\usepackage{diagbox}
\usepackage{lipsum}
\usepackage{amsfonts}
\usepackage{graphicx}
\usepackage{epstopdf}
\usepackage{algorithmic}
\ifpdf
  \DeclareGraphicsExtensions{.eps,.pdf,.png,.jpg}
\else
  \DeclareGraphicsExtensions{.eps}
\fi

\usepackage{multirow}
\newcommand{\minitab}[2][l]{\begin{tabular}{#1}#2\end{tabular}}

\newsiamremark{remark}{Remark}
\newsiamremark{hypothesis}{Hypothesis}
\crefname{hypothesis}{Hypothesis}{Hypotheses}
\newsiamthm{claim}{Claim}

\newtheorem{example}{Example} 
\numberwithin{equation}{section}
\headers{Parameter-Robust Preconditioners for A TPE Model}{M. Cai, M. Kuchta, J. Li, Z. Li, and K.-A. Mardal}

\title{Parameter-Robust Preconditioners for A Four-Field Thermo-Poroelasticity Model\thanks{Submitted to the editors DATE.
}}

\author{Mingchao Cai\thanks{Department of Mathematics, Morgan State University, Baltimore,
MD 21251, USA
  (\email{Mingchao.Cai@morgan.edu}, 
  ).}
\and
Miroslav Kuchta\thanks{
Department of Numerical Analysis and Scientific Computing, Simula Research Laboratory, Oslo, 0164 Norway. 
  (\email{miroslav@simula.com}, \email{kent-and@simula.no}).}
\and 
   Jingzhi Li\thanks{Department of Mathematics, Southern University of Science and
Technology, Shenzhen 518055, P.R. China
  (\email{li.jz@sustech.edu.cn}, \email{liziliang@whu.edu.cn}).}
\and 
     Ziliang Li\footnotemark[4]
\and Kent-Andre Mardal\footnotemark[3]
}

\usepackage{amsopn}

\ifpdf
\hypersetup{
  pdftitle={Parameter-Robust Preconditioners for A Four-Field Thermo-Poroelasticity Model},
  pdfauthor={M. Cai, M. Kuchta, J. Li, Z. Li, and K.-A. Mardal}
}
\fi

\begin{document}

\maketitle

\begin{abstract}
We study a thermo-poroelasticity model which describes the interaction between the deformation of an elastic porous material and fluid flow under non-isothermal conditions. The model involves several parameters that can vary significantly in practical applications, posing a challenge for developing discretization techniques and solution algorithms that handle such variations effectively. We propose a four-field formulation and apply a conforming finite element discretization. The primary focus is on constructing and analyzing preconditioners for the resulting linear system. Two preconditioners are proposed: one involves regrouping variables and treating the 4-by-4 system as a 2-by-2 block form, while the other is directly constructed from the 4-by-4 coupled operator. Both preconditioners are demonstrated to be robust with respect to variations in parameters and mesh refinement. Numerical experiments are presented to demonstrate the effectiveness of the proposed preconditioners and validate their theoretical performance under varying parameter settings.
\end{abstract}

\begin{keywords}
  thermo-poroelasticity, parameter-robust preconditioners, finite element methods
\end{keywords}

\begin{MSCcodes}
  65M60, 65F08, 65F10
\end{MSCcodes}

\section{Introduction}
 In this work, we explore numerical methods for the thermo-poroelasticity model, which describes the coupled interaction between non-isothermal fluid flow and the deformation of porous materials. The model is a coupling of Biot's equation \cite{terzaghi1943theoretical, biot1941general, biot1955theory} with the energy conservation equation, specifically addressing the interaction within poroelasticity, encompassing mechanical effects, fluid flow, and heat transfer in porous media. Let $\Omega \subset \mathbb{R}^d$, $d =2$ or $3$ be an open, bounded domain with a Lipschitz polyhedral boundary $\partial \Omega$, and let $J = (0, t_f)$ denote the time interval, where $t_f > 0$ represents the final time. The resulting space-time domain is $\Omega \times J$. The general nonlinear thermo-poroelasticity model \cite{brun2018upscaling, brun2019well, brun2020monolithic} is formulated as: finding $(\bm{u}, p, T)$ such that
\begin{equation}\label{general TP_Model}
\begin{aligned} 
       -\nabla\cdot( 2\mu\bm\varepsilon(\bm u)+\lambda\nabla\cdot\bm u\bm I )+\alpha\nabla p+\beta\nabla T
                      &=\bm f,\quad\text{in }\Omega\times J ,\\
\frac\partial{\partial t}(c_0 p-b_0 T+\alpha\nabla\cdot\bm u)-\nabla\cdot(\bm K\nabla p)
                      &=g,\quad\text{in }\Omega\times J ,\\
\frac\partial{\partial t}(a_0T-b_0p+\beta\nabla\cdot\bm u) 
-C_f(\bm K\nabla p)\cdot\nabla T
-\nabla\cdot(\bm\Theta\nabla T)
                      &=   H,\quad\text{in }\Omega\times J .\\
\end{aligned}
\end{equation}
Here, the operator $\bm\varepsilon$  is defined as $\bm\varepsilon(\bm u)=\frac{1}{2}(\nabla\bm u+\nabla\bm u^T)$ and $\bm I$  is the identity tensor.
The variables \(\bm u\), \(p\), and \(T\) represent the displacement, fluid pressure, and temperature distribution with respect to a reference value, respectively. The right-hand side terms \(\bm f\), \(g\), and \(  H \) correspond to the body force, mass source, and heat source, respectively. The three equations in \eqref{general TP_Model} reflect the principles of momentum conservation, mass conservation, and energy conservation. For a comprehensive derivation of the model, readers are referred to \cite{brun2018upscaling}, which incorporates an additional nonlinear convective term involving \(\frac{\partial}{\partial t}\bm u\) and \(\nabla T\). However, as indicated in \cite{brun2020monolithic, brun2019well, van2019thermoporoelasticity}, this term is considered negligible in comparison to the heat convection term associated with the Darcy velocity. The parameters \(\bm K = (K_{ij})^d_{ij=1}\) and \(\bm\Theta = (\Theta_{ij})^d_{ij=1}\) are matrices determined by the medium's permeability and the fluid viscosity. Detailed descriptions of other parameters, including their physical meanings and corresponding units, can be found in Table \ref{tab:parameters}, where the problem parameters for \eqref{general TP_Model} are summarized. Additionally, the bulk modulus of the porous material, \(K_D\), is related to the two Lam{\'e} parameters \(\lambda\) and \(\mu\) by the formula \(K_D := d^{-1}(d\lambda + 2\mu)\). For further discussion on parameter relations, we refer to \cite{antonietti2023discontinuous}.

In the following, the system \eqref{general TP_Model} is completed by boundary conditions
 \begin{equation}\label{BC}
 \bm u =\bm 0,~p = 0,
  \text{ and } T = 0, \text{ on } \partial\Omega\times J,
\end{equation}
and initial conditions
\begin{equation}\label{IC}
 \bm u(\cdot,0)=\bm u^0,~p(\cdot,0)=p^0,~T(\cdot,0)=T^0, \text{ in } \Omega\times\{0\},  
\end{equation}
for some known functions $\bm u^0, p^0$ and $T^0$.

\begin{table}[]
\centering
\caption{Thermo-Poroelasticity Model Parameters and Their Units}
\begin{tabular}{@{}lll@{}}
\toprule
Notation & Quantity & Unit  \\ \midrule
$a_0$   & thermal capacity & $Pa/K^2$  \\ %
$b_0$   & thermal dilatation coefficient & $K^{-1}$ \\
$c_0$   & specific storage coefficient & $Pa^{-1}$ \\
$\alpha$& Biot-Willis constant&-\\
$\beta$ & thermal stress coefficient & $Pa/K$\\ %
$C_f$   & fluid volumetric heat capacity divided  & $Pa/K^2 $\\
&   by reference temperature & \\%
$\mu,\lambda$ & Lam\'e coefficients & $Pa$ \\ %
$\bm K$& permeability divided by fluid viscosity &$m^2/(Pas)$\\
$\bm\Theta $& effective thermal conductivity &$m^2Pa/(K^2s)$\\
\bottomrule
\end{tabular}\label{tab:parameters}
\end{table}


Recent advancements in numerical analysis for thermo-poroelasticity models include studies on solution existence, uniqueness, and energy estimates for nonlinear problems, as in \cite{brun2019well}. 
By introducing three auxiliary variables and performing multiple substitutions, the authors in \cite{chen2022multiphysics, ge2023analysis} transformed the original model into a four-field formulation, enabling multiphysics finite element methods applicable to both linear and nonlinear convective transport scenarios. Additionally, the numerical approaches proposed in these works demonstrate robustness with respect to the parameter $\lambda$.
Mixed finite element and hybrid techniques have been applied to these models \cite{chen2022multiphysics, zhang2022galerkin, zhang2024coupling}. Robust discontinuous Galerkin methods for fully coupled nonlinear models are detailed in \cite{bonetti2024robust, antonietti2023discontinuous}. Iterative coupling techniques include sequential iteration methods for linear problems \cite{kim2018unconditionally} and splitting schemes for decoupling poroelasticity and thermoelasticity \cite{kolesov2014splitting}. A five-field formulation incorporating heat and Darcy fluxes is presented in \cite{brun2020monolithic}, with both monolithic and decoupled schemes. Recently, reduced-order modeling for the linear thermo-poroelasticity model was introduced in \cite{ballarin2024projection} to enhance the efficiency of decoupled iterative solutions. 

Numerical solutions for the thermo-poroelasticity model are challenging, partly due to significant variations in model parameters across applications. For example, in macroscopic thermo-poroelasticity models within rock mechanics \cite{van2019thermoporoelasticity}, permeability can vary drastically, ranging from \(10^{-5}\) to \(10^{-20}\) \(m^2\). Similarly, in non-isothermal fluid flow models through deformable porous media \cite{van2020mathematical}, the Lam\'{e} coefficients are on the order of \(O(1)\) in magnitude, implying that the bulk modulus is also \(O(1)\) in \(Pa\), while permeability is around \(O(10^{-4})\) \(m^2\). In contrast, models that involve rigid one-dimensional fluid cavities \cite{selvadurai2016thermo} exhibit bulk moduli on the order of \(GPa\), with permeability varying around \(O(10^{-19})\) \(m^2\). Additionally, spatial and temporal discretizations introduce discretization parameters, further complicating the problem. Effective numerical methods must therefore remain robust against significant variations in both model and discretization parameters. For example, \cite{hong2017parameter, lee2017parameter, piersanti2021parameter, boon2021robust} developed parameter-robust finite element discretizations and uniform block-diagonal preconditioners for poroelasticity models. However, the challenges of ensuring parameter robustness and effective preconditioning for thermo-poroelasticity models remain unaddressed. For large discrete systems solved via iterative methods, the convergence rate heavily depends on the condition numbers of preconditioned systems \cite{axelsson2012stable, haga2012parallel}. Although preconditioning techniques for poroelasticity have been studied in \cite{lee2017parameter, hong2017parameter, budivsa2021block}, no effort has focused on thermo-poroelasticity models.

This paper addresses a key gap by proposing a stable finite element method for the thermo-poroelasticity model and developing the corresponding preconditioners. The primary aim is to ensure that the condition number of the preconditioned systems remains bounded across a wide range of model parameters. To focus on the preconditioner design, we omit the nonlinear term \( C_f(\bm K\nabla p)\cdot\nabla T \). Similar to the quasi-static Biot's model, Poisson locking  \cite{oyarzua2016locking, lee2017parameter, antonietti2023discontinuous} occurs when the Lam\'{e} constant $\lambda$ approaches infinity. To address this, we introduce an auxiliary variable, $\xi = -\lambda\nabla\cdot\bm{u} + \alpha p + \beta T$, which captures the volumetric contribution to the total stress. 
This reformulation transforms the original three-field formulation into a symmetric four-field model, effectively mitigating Poisson locking and enabling the application of the operator preconditioning framework from \cite{mardal2011preconditioning,hong2021new}. 
Upon discretization, the four-field formulation results in a large, indefinite linear system. To address this, we analyze the system's stability within weighted Hilbert spaces and apply operator preconditioning techniques. By defining appropriate norms, we prove that the constants in the boundedness and inf-sup conditions are independent of the model parameters, ensuring the framework's uniform robustness under minimal assumptions. To validate the effectiveness of the proposed preconditioners, we conduct numerical experiments using both exact LU decomposition and an inexact algebraic multigrid (AMG) solver for the elliptic operators. These experiments demonstrate the preconditioners' robustness with respect to both the physical parameters of the model and the discretization parameters. 

The paper is structured as follows. \Cref{sec:PDE model} presents a linear thermo-poroelastic model and its four-field formulation, leading to a symmetric indefinite linear system after time discretization. \Cref{sec:preconditioner} introduces two preconditioners and examines their robustness with respect to physical parameters. \Cref{sec:dis} details the construction of a conforming finite element discretization and parameter-robust preconditioners. \Cref{sec:Numerical results} provides numerical experiments to validate the theoretical findings.

\section{Parameter-dependent systems}\label{sec:PDE model}
Throughout this paper, we adopt the following definitions and notations. Let \( L^p(\Omega) \) denote the standard Lebesgue space on \( \Omega \) with index \( p \in [1, \infty] \). In particular, for \( p = 2 \), \( L^2(\Omega) \) represents the space of square-integrable functions on \( \Omega \), equipped with the inner product \( ( \cdot, \cdot ) \) and norm \( \| \cdot \|_0 \). For Sobolev spaces, we define \( W^{m,p}(\Omega) = \{ u \,|\, D^\alpha u \in L^p(\Omega), 0 \leq \alpha \leq m, \|u\|_{W^{m,p}} < \infty \} \), and write \( H^m(\Omega) \) as shorthand for \( W^{m,2}(\Omega) \), with \( \|\cdot\|_{H^m(\Omega)} \) representing the associated norm. We denote \( H^m_0(\Omega) \) as the subspace of \( H^m(\Omega) \) with a vanishing trace on \( \partial\Omega \), and \( H^m_{0,\Gamma}(\Omega) \) as the subspace of \( H^m(\Omega) \) with a vanishing trace on \( \Gamma \subset \partial\Omega \). Specifically, for \( m = 2 \), we use \( \|\cdot\|_1 \) in place of \( \|\cdot\|_{H^m(\Omega)} \) or \( \|\cdot\|_{W^{m,2}(\Omega)} \).
For a Banach space \( X \), we define \( L^p(J; X) = \{ v \,|\, (\int_J \|v\|_X^p)^{\frac{1}{p}} < \infty \} \) and \( H^1(J; X) = \{ v \,|\, (\int_J (\|v\|_0^2 + \|\partial_t v\|_0^2) \, dt)^{\frac{1}{2}} < \infty \} \). 


\subsection{A simplified linear thermo-poroelastic model}

As stated earlier, the nonlinear term \( C_f(\bm K\nabla p)\cdot\nabla T \) in \eqref{general TP_Model} is omitted to simplify the discussion. Therefore, we focus on the following simplified linear thermo-poroelastic model.
\begin{equation}\label{TP_Model}
\begin{aligned} 
       -\nabla\cdot( 2\mu\bm\varepsilon(\bm u)+\lambda\nabla\cdot\bm u\bm I )+\alpha\nabla p+\beta\nabla T
                      &=\bm f,\quad\text{in }\Omega\times J ,\\
\frac\partial{\partial t}(c_0 p-b_0 T+\alpha\nabla\cdot\bm u)-\nabla\cdot(\bm K\nabla p)
                      &=g,\quad\text{in }\Omega\times J ,\\
\frac\partial{\partial t}(a_0T-b_0p+\beta\nabla\cdot\bm u) -\nabla\cdot(\bm\Theta\nabla T)
                      &=H ,\quad\text{in }\Omega\times J .\\
\end{aligned}
\end{equation}

We retain the Dirichlet boundary conditions (\ref{BC}) and the initial condition (\ref{IC}). In practical scenarios, nonhomogeneous Dirichlet and Neumann boundary conditions are commonly encountered. The analysis performed for homogeneous boundary conditions can be straightforwardly extended to accommodate these cases.
We assume that $g,h\in L^2(J;L^2(\Omega))$ and $\bm f\in [H^1(J;L^2(\Omega))]^d$.
And we assume that the initial conditions satisfy $p^0,T^0\in H_0^1(\Omega)$ and $\bm u^0\in [H_0^1(\Omega)]^d$.
Furthermore, as outlined in \cite{brun2020monolithic}, we adopt the following \textbf{assumptions} for the model parameters throughout this paper (similar assumptions are also discussed in \cite{chen2022multiphysics, zhang2024coupling}):

(A1) Assume that \(\bm{K} = K\bm{I}\) and \(\bm\Theta = \theta\bm{I}\), where \(K\) and \(\theta\) are positive constants bounded both above and below, and \(\bm{I}\) denotes the identity matrix.

(A2) The constants $\lambda$  and $\mu$ are strictly positive constants.

(A3) The coefficients $\alpha$, $\beta$, $a_0, b_0, c_0 \geq 0$ and $a_0, c_0 \geq b_0$.

For convenience, we introduce the following parameter transformations:  
\[
t_K = \Delta t K, ~ t_\theta = \Delta t \theta, ~ c_\alpha = \left(c_0 + \frac{\alpha^2}{\lambda}\right), ~ c_{\alpha\beta} = \left(\frac{\alpha \beta}{\lambda} - b_0\right), ~ c_\beta = \left(a_0 + \frac{\beta^2}{\lambda}\right).
\]  
Additionally, to further simplify, we assume \(2\mu = 1\).
Note that from the following proof, if we keep the coefficient $2\mu$, all the conclusions 
still hold true. 
We have adopted this assumption solely for the sake of formal expediency. 

We now introduce the four-field formulation for the linear component of \eqref{TP_Model}. More clearly, following the methodology for handling Biot's problem in \cite{lee2017parameter, oyarzua2016locking}, we define an auxiliary variable to represent the volumetric contribution to the total stress:  
$$  
\xi = -\lambda \nabla \cdot \bm{u} + \alpha p + \beta T,
$$  
which is commonly referred to as the pseudo-total pressure \cite{antonietti2023discontinuous}. Substituting this variable into equation \eqref{TP_Model} and using parameter transformations above, the four-field thermo-poroelasticity problem can be expressed as:
\begin{equation}\label{TP_Model_four}
\begin{aligned} 
       - \nabla\cdot  \bm\varepsilon(\bm u) +\nabla\xi 
                      &=\bm f, \\
      -\lambda \nabla\cdot\bm u-\xi+\alpha p+\beta T
                      &=0, \\
\frac\partial{\partial t}(-\frac\alpha\lambda\xi+c_\alpha p+c_{\alpha\beta} T)
                 -\nabla\cdot( K\nabla p)
                      &=g, \\
\frac\partial{\partial t}(-\frac\beta\lambda\xi +c_{\alpha\beta} p +c_\beta T)
                  -\nabla\cdot( \theta\nabla T)
                      &= H . \\
\end{aligned}
\end{equation}

For the time discretization, we use an equidistant partition of the interval \([0, t_f]\) with a constant step size \(\Delta t\). At any time step, the relationship is given by \(t_n = t_{n-1} + \Delta t\). Using the backward Euler method, we solve for \((\bm u^n, \xi^n, p^n, T^n)\) at each time step based on the equation \eqref{TP_Model_static}, as follows:
\begin{equation}\label{TP_Model_static}
\begin{aligned} 
       - \nabla\cdot  \bm\varepsilon(\bm u^n) +\nabla\xi^n 
                      &=\bm f^n, \\
      -\lambda \nabla\cdot\bm u^n-\xi^n+\alpha p^n+\beta T^n
                      &=0, \\
-\frac\alpha\lambda \xi^n+c_\alpha p^n
                 -\Delta t\nabla\cdot( K\nabla p^n)+c_{\alpha\beta}T^n
                      &=\tilde g^n, \\
-\frac\beta\lambda \xi^n +c_{\alpha\beta} p^n +c_\beta T^n
                  -\Delta t\nabla\cdot(\theta\nabla T^n)
                      &=\tilde {H}^n, \\
\end{aligned}
\end{equation}
where 
$$
\bm u^n=\bm u(\bm x;t_n),~ \xi^n=\xi(\bm x;t_n),~p^n=p(\bm x;t_n),~\bm f^n=\bm f(\bm x;t_n),
$$   
$$
\tilde g^n=\Delta tg(\bm x;t_n)-\frac\alpha\lambda \xi(\bm x;t_{n-1})
        +c_\alpha p(\bm x;t_{n-1})
                      +c_{\alpha\beta}T(\bm x;t_{n-1}),
$$
and
$$
\tilde { H }^n=\Delta t h(\bm x;t_n)-\frac\beta\lambda \xi(\bm x;t_{n-1})
        +c_{\alpha\beta} p(\bm x;t_{n-1})
                     +c_\beta T(\bm x;t_{n-1}).
$$

Since our focus is on the system of linear equations at a specific time step \(n\), we simplify the notation by omitting the time-step subscript. Thus, we replace \(\bm u^n, \xi^n, p^n, T^n, \bm f^n, \tilde g^n\), and \(\tilde { H }^n\) with \(\bm u, \xi, p, T, \bm f, \tilde g\), and \(\tilde {H}\), respectively. This results in the following system of equations:
\begin{equation}\label{TP_Model_param}
\begin{aligned} 
       -  \nabla\cdot  \bm\varepsilon(\bm u) +\nabla\xi
                      &=\bm f, \\
      - \nabla\cdot\bm u-\lambda^{-1}\xi+\frac{\alpha}{\lambda} p+\frac{\beta}{\lambda} T
                      &=0, \\
\frac{\alpha}{\lambda} \xi-c_\alpha p  +  t_K\nabla\cdot( \nabla p)-c_{\alpha\beta}T
                      &=-\tilde g, \\
\frac{\beta}{\lambda} \xi - c_{\alpha\beta} p
-c_\beta T + t_\theta\nabla\cdot( \nabla T)
                      &=-\tilde {H}. \\
\end{aligned}
\end{equation}

We define the following functional spaces:  
\[
\bm{V} = [H^1_0(\Omega)]^d, \quad Q = L^2(\Omega), \quad W = H_0^1(\Omega).
\]
The corresponding continuous variational formulation for \eqref{TP_Model_param} 
is stated as follows:  
find \((\bm u, \xi, p, T) \in \bm V \times Q \times W \times W\)  
such that, for all \((\bm v, \phi, q, S) \in \bm V \times Q \times W \times W\), the following holds:  
\begin{equation}\label{TP_Model_var}
\begin{aligned} 
 (\bm\varepsilon(\bm u),\bm\varepsilon(\bm v)) - (\nabla\cdot\bm v,\xi)&=(\bm f,\bm v) ,   \\
  -(\nabla\cdot\bm u,\phi)- ({\lambda^{-1}}\xi,\phi)+(\frac{\alpha}{\lambda}p,\phi)+(\frac{\beta}{\lambda}T,\phi)&=0 ,\\
 (\frac{\alpha}{\lambda}\xi,q)-(c_\alpha p ,q)- (t_K\nabla p,\nabla q)- (c_{\alpha\beta}T ,q)
                &=(-\tilde g,q),\\
 (\frac{\beta}{\lambda}\xi ,S)-(c_{\alpha\beta}p ,S)-(c_\beta T ,S)   
  -  (t_\theta\nabla T,\nabla S)&=(-\tilde {H}, S).\\
\end{aligned}
\end{equation}

\section{Preconditioners and parameter-robust stability}\label{sec:preconditioner}


Let \( X \) be a separable, real Hilbert space equipped with a norm \(\|\cdot\|_X\), to be defined later, and let its dual space be denoted by \( X^* \). Consider an operator \(\mathcal{A}: X \to X^*\), which is an invertible, symmetric isomorphism and belongs to the space of bounded linear operators, \(\mathcal{L}(X, X^*)\). Given \(\mathcal{F} \in X^*\), the goal is to find \(x \in X\) such that:  
\begin{equation*}  
\mathcal{A} x = \mathcal{F}.  
\end{equation*}  
The operator norm of \(\mathcal{A}\) in \(\mathcal{L}(X, X^*)\) is defined as:  
\begin{equation*}  
\|\mathcal{A}\|_{\mathcal{L}(X, X^*)} = \sup_{x \in X} \frac{\|\mathcal{A}x\|_{X^*}}{\|x\|_X}.  
\end{equation*}  
To improve computational efficiency, a preconditioner \(\mathcal{B}^{-1} \in \mathcal{L}(X^*, X)\), which is a symmetric isomorphism, is introduced. The preconditioned problem is written as:  
\begin{equation*}  
\mathcal{B}^{-1}\mathcal{A} x = \mathcal{B}^{-1}\mathcal{F}.  
\end{equation*}  
The convergence rate of iterative methods, such as the MinRes method for solving this problem, depends on the condition number \(\kappa(\mathcal{B}^{-1}\mathcal{A})\), defined as:  
\begin{equation*}  
\kappa(\mathcal{B}^{-1}\mathcal{A}) = \|\mathcal{B}^{-1}\mathcal{A}\|_{\mathcal{L}(X, X)} \|(\mathcal{B}^{-1}\mathcal{A})^{-1}\|_{\mathcal{L}(X, X)}.  
\end{equation*}  

For a parameter-dependent operator \(\mathcal{A}_\epsilon \in \mathcal{L}(X_{\epsilon}, X^*_{\epsilon})\), the objective is to design a preconditioner \(\mathcal{B}_\epsilon\) such that the condition number of the preconditioned system is robust with respect to a set of parameters \(\epsilon\) (including \(\mu\), \(\lambda\), \(\alpha\), \(\beta\), \(a_0\), \(b_0\), \(c_0\), \(K\), and \(\theta\)). We assume that appropriate function spaces \(X_\epsilon\) and \(X_\epsilon^*\) can be identified such that the following operator norms:  
\begin{equation*}  
\|\mathcal{A}_\epsilon\|_{\mathcal{L}(X_\epsilon, X_\epsilon^*)}, \quad  
\|(\mathcal{A}_\epsilon)^{-1}\|_{\mathcal{L}(X_\epsilon^*, X_\epsilon)}, \quad  
\|\mathcal{B}_\epsilon\|_{\mathcal{L}(X_\epsilon^*, X_\epsilon)}, \quad  
\|(\mathcal{B}_\epsilon)^{-1}\|_{\mathcal{L}(X_\epsilon, X_\epsilon^*)}  
\end{equation*}  
remain uniformly bounded, independent of the parameters \(\epsilon\).  
Under these assumptions, the condition number \(\kappa(\mathcal{B}_\epsilon^{-1} \mathcal{A}_\epsilon)\) will also be uniformly bounded, regardless of the values of \(\epsilon\). 

We note that 
the system \eqref{TP_Model_param} can be expressed in operator form as follows:
\begin{equation}\label{TP matrix}
\begin{bmatrix}
- \text{div }\bm\varepsilon     & \nabla             &  0                  &  0               \\             
-\text{div}& -\lambda^{-1}I    &\frac{\alpha}{\lambda}I & \frac{\beta}{\lambda}I\\
0          &\frac{\alpha}{\lambda}I     &-c_\alpha  I+  t_K\text{div}( \nabla)&-c_{\alpha\beta}  I\\
0          & \frac{\beta}{\lambda} I       &-c_{\alpha\beta}I   &-c_\beta  I+ t_\theta\text{div}(\nabla)\\
\end{bmatrix}
\begin{bmatrix}
 \bm u \\ \xi \\ p \\ T
\end{bmatrix}
=
\begin{bmatrix}
 \bm f\\ 0 \\ -\tilde g \\ -\tilde {H}
\end{bmatrix},
\end{equation}  
where $I$ is an identity operator. 
Let the coefficient matrix of the system be denoted by the operator $\mathcal{A}$ and the right-hand side part be denoted by $\mathcal F$. For simplicity, we rewrite the system \eqref{TP matrix} as:  
\begin{equation*}  
\mathcal{A} (\bm{u}, \xi, p, T) = \mathcal{F}.  
\end{equation*}  
It is easy to test 
that $\mathcal A$ is a symmetric linear operator with respect to the Dirichlet boundary condition:  
\begin{equation*}\label{bilinear A}
\begin{aligned} 
(\mathcal A(\bm u ,\xi, p, T),(\bm v, \phi, q, S))
&=((\bm u ,\xi, p, T),\mathcal A(\bm v, \phi, q, S))\\
&=( \bm\varepsilon(\bm u),\bm\varepsilon(\bm v)) - (\nabla\cdot\bm v,\xi)  
  -(\nabla\cdot\bm u,\phi)- ({\lambda^{-1}}\xi,\phi)\\
  &+(\frac{\alpha}{\lambda}p,\phi)+(\frac{\beta}{\lambda}T,\phi)  
   +(\frac{\alpha}{\lambda}\xi ,q)-(c_\alpha p ,q)- (t_K\nabla p,\nabla q)
   \\
   & -(c_{\alpha\beta}T ,q) +(\frac{\beta}{\lambda}\xi ,S)-(c_{\alpha\beta}p ,S)-(c_\beta T,S)   
  -  ( t_\theta\nabla T,\nabla S) .
\end{aligned}
\end{equation*} 

Before we define appropriate parameter-dependent norms for $\mathcal A$, we first recall the Lam\'e problem in linear elasticity:
Find  $\bm u : \Omega  \rightarrow \mathbb{R}^d$, $p : \Omega \rightarrow \mathbb{R}$ for 
\begin{equation}\label{ex: elasticity}
\begin{aligned}  
&-\text{div} (\bm\varepsilon(\bm u)) + \nabla p = \bm f, \\
&-\text{div}\bm u-\frac{1}{\lambda}p=g,
\end{aligned}  
\end{equation}
with $\bm u|_{\partial\Omega} = \bm{0}$. Here, $1 \leq \lambda < +\infty$ is a positive constant 
and $\bm\varepsilon(\bm u)$ is the symmetric
gradient of $\bm u$. 
When the three variables \((\xi, p, T)\) are grouped together, the system in (\ref{TP matrix}) resembles the problem in \eqref{ex: elasticity}. The variational form of problem \eqref{ex: elasticity} reads as: find \(\bm u \in [H_0^1(\Omega)]^d\) and \(p \in L^2(\Omega)\) such that:  
\begin{equation}\label{ex: elasticity var}  
\begin{aligned}  
    &(\bm\varepsilon(\bm u), \bm\varepsilon(\bm v)) - (p, \text{div}\bm v) = (f, \bm v),  
    \quad \forall \bm v \in [H_0^1(\Omega)]^d, \\  
    &-(\text{div}\bm u, q) - \frac{1}{\lambda}(p, q) = (g, q), \quad \forall q \in L^2(\Omega).  
\end{aligned}  
\end{equation}  
In this saddle point problem, the stabilizing term \(\frac{1}{\lambda} (p, q)\) weakens as \(\lambda\) increases. In the limiting case where \(\lambda = +\infty\), the system becomes unstable under standard norms such as \((\|\bm u\|_1^2 + \|p\|_0^2)^{1/2}\). This instability arises because \(\text{div}[H_0^1(\Omega)]^d \subsetneq L^2(\Omega)\). 
To ensure \(\lambda\)-independent stability for the system, it is important to note that \(\text{div}[H_0^1(\Omega)]^d\) controls only the \(L^2(\Omega)\) norm of the zero-mean component of \(p\), leaving the mean value part of \(p\) uncontrolled without the stabilizing term.  
For any \(\phi \in L^2(\Omega)\), its mean-value and mean-zero components are defined as  
\begin{equation}\label{def of phi_0}  
\phi_m := P_m\phi \quad \text{and} \quad \phi_0 := \phi - \phi_m,  
\end{equation}  
where  
\[
P_m \phi := \left(\frac{1}{|\Omega|} \int_\Omega \phi \, dx \right) \chi_{\Omega},  
\]  
with \(\chi_{\Omega}\) representing the characteristic function of \(\Omega\), and \(|\Omega|\) denoting the Lebesgue measure of \(\Omega\). 
To achieve \(\lambda\)-robust stability for the problem in (\ref{ex: elasticity var}), the appropriate Hilbert space with a \(\lambda\)-independent norm is given by:  
\[
\|(\bm u, p)\|^2 = \|\bm u\|_1^2 + \frac{1}{\lambda}\|p_m\|_0^2 + \|p_0\|_0^2.
\]  
Such a formulation naturally suggests considering \(p \in \lambda^{-1/2} L^2(\Omega) \cap L_0^2(\Omega)\).  
This leads to a \(\lambda\)-robust preconditioner for the problem (\ref{ex: elasticity}) as proposed in \cite{lee2017parameter}:  
\begin{equation}\label{eq:pre_elasticity}
\begin{bmatrix}  
-\Delta^{-1} & 0 \\  
0 & (\frac{1}{\lambda}I_m + I_0)^{-1}  
\end{bmatrix}.
\end{equation}
Here, \(I_0\) represents the Riesz map from \(L^2(\Omega)\) to the dual of \(L^2_0(\Omega)\), while \(I_m\) is the corresponding operator mapping \(L^2(\Omega)\) to the dual of the complement of \(L^2_0(\Omega)\).

\subsection{A parameter-robust preconditioner by regrouping variables}

 By grouping the variables into \(\bm{u}\) and \((\xi, p, T)\), and utilizing the product space \(\bm{V} \times (Q \times W \times W)\), the system \eqref{TP matrix} can be reformulated into the standard saddle point structure:
\begin{equation} \label{mathcal A blocks}   
\mathcal{A} =  
\begin{bmatrix}  
A_0 & B_0^* \\  
B_0 & -C_0  
\end{bmatrix},  
\end{equation}  
where
$
A_0= - \text{div}\bm\varepsilon, B_0=(-\text{div}, 0, 0)^T,  
$
and
$$
C_0=\begin{bmatrix}           
         & \lambda^{-1}I    &-\frac{\alpha}{\lambda}I  & -\frac{\beta}{\lambda}I\\
         &-\frac{\alpha}{\lambda}I     &c_\alpha  I-  t_K\text{div}( \nabla)&c_{\alpha\beta}  I\\
         &- \frac{\beta}{\lambda}  I        &c_{\alpha\beta}I      &c_\beta  I- t_\theta\text{div}(\nabla)
\end{bmatrix},
$$ 
and $B_0^*=(\nabla , 0, 0)$ is the adjoint operator of $B_0$.
Building on the parameter-robust preconditioner developed for saddle point problems with a penalty term, as detailed in \cite{boon2021robust, braess1996stability, hong2021new}, and leveraging the block-processing method for multi-network poroelasticity problems from \cite{piersanti2021parameter}, we adapt and extend these methodologies to tackle the thermo-poroelasticity problem presented in this paper.

It is natural to use the block diagonal operator
\begin{equation*} 
\begin{bmatrix}
A_0^{-1}  & 0           \\             
0       &  (C_0+B_0A_0^{-1}B_0^*)^{-1}
\end{bmatrix},
\end{equation*}
or its approximation to construct block preconditioners for $\mathcal{A}$.
 To this end, we consider a parameter-dependent norm for the thermo-poroelasticity problem as follows:  
\begin{equation}\label{3blocks norm}  
\begin{aligned}  
\|(\bm{u}, \xi, p, T)\|_{\mathcal{B}_1}^2 = &  \|\bm\varepsilon(\bm{u})\|_0^2 + \frac{1}{\lambda}\|\xi\|_0^2 + \|\xi_0\|_0^2 - 2(\frac{\alpha}{\lambda}p, \xi) - 2(\frac{\beta}{\lambda}T, \xi) \\  
& + c_\alpha\|p\|_0^2 + t_K\|\nabla p\|_0^2 + c_\beta\|T\|_0^2 + t_\theta\|\nabla T\|_0^2 + 2(c_{\alpha\beta}p, T),  
\end{aligned}  
\end{equation}  
where \(\xi_0\) is the mean value zero of \(\xi\), as defined in \eqref{def of phi_0}.

The norm defined in \eqref{3blocks norm} contains some negative terms, requiring a demonstration to confirm that it qualifies as a valid norm. To address this, we introduce a bilinear form \((\mathcal{B}_1 \cdot, \cdot)\), where the operator \(\mathcal{B}_1\) is defined as:  
\begin{equation} \label{operator B1}  
\mathcal{B}_1 =   
\begin{bmatrix}  
-  \text{div }\bm\varepsilon & 0 & 0 & 0 \\  
0 & \lambda^{-1}I + I_0 & -\frac{\alpha}{\lambda}I & -\frac{\beta}{\lambda}I \\  
0 & -\frac{\alpha}{\lambda}I & c_\alpha I - t_K \text{div}(\nabla) & c_{\alpha\beta}I \\  
0 & -\frac{\beta}{\lambda}I & c_{\alpha\beta}I & c_\beta I - t_\theta \text{div}(\nabla)  
\end{bmatrix},  
\end{equation}  
and is subject to the same boundary conditions as those in (\ref{BC}).
We recall that \(I\) represents the Riesz map from \(Q\) to its dual \(Q^*\) and \(I_0\) maps \(Q\) to the dual of \(Q \cap L_0^2(\Omega)\). 
It should be noted that, unlike the preconditioner \eqref{eq:pre_elasticity} used for the Lam{\'e} problem, the preconditioner in \eqref{eq:precond3x3} employs the operator \(\lambda^{-1}I + I_0\) in the $(2,2)$ block instead of \(\lambda^{-1}I_m + I_0\). For \(\lambda \geq 1\), however, these two operators are spectrally equivalent \cite{lee2017parameter}. For further implementation details, we refer the readers to \cite{lee2017parameter, mardal2011preconditioning}. 
Additionally, we highlight that the lower \(3 \times 3\) block of \(\mathcal{B}_1\) is connected to the preconditioner used in the multiple-network poroelastic model \cite{piersanti2021parameter}, where diagonalization by congruence was applied in order simplify evaluation of the discrete preconditioner in terms of multilevel methods. In contrast, our approach deals directly with the \(3 \times 3\) operator. In particular, we show in \Cref{rmrk:amg} that the block is amenable to algebraic multigrid.


Defining an operator matrix 
$
\mathcal{B}_{22}=\begin{bmatrix} 
         0    &0  & 0 &0\\
         0    &I_0  & 0 &0\\
         0     &0  &0 &0\\
         0     &0  &0 &0
\end{bmatrix}
$and we can split $\mathcal{B}_1$ as
$$\mathcal{B}_1=\tilde{\mathcal{B}_1}+\mathcal{B}_{22}, $$
where $\tilde{\mathcal{B}_1}=\mathcal{B}_1-\mathcal{B}_{22}$.
Given the assumption \(c_0 - b_0 \ge 0\) and \(a_0 - b_0 \ge 0\), we establish:  
\[
\begin{aligned}  
(c_0p, p) + (a_0T, T) - 2(b_0p, T) &\geq (c_0p, p) + (a_0T, T) - (b_0p, p) - (b_0T, T) \\  
&= (C_p p, p) + (C_T T, T) \ge 0,  
\end{aligned}  
\]  
for any \((p, T) \neq (0, 0)\), where \(C_p = c_0 - b_0\) and \(C_T = a_0 - b_0\).  
Using the definitions of \(c_\alpha\), \(c_\beta\), and \(c_{\alpha\beta}\), along with Green's formula and the boundary conditions in (\ref{BC}), we derive:  
\[
\begin{aligned}  
&(\tilde{\mathcal{B}_1}(\bm{u}, \xi, p, T), (\bm{u}, \xi, p, T)) \\  
=&  \|\bm{\varepsilon}(\bm{u})\|_0^2 + \left(\frac{1}{\lambda}\xi, \xi\right) - 2\left(\frac{\alpha}{\lambda}p, \xi\right) - 2\left(\frac{\beta}{\lambda}T, \xi\right)  
+ (c_\alpha p, p) + (c_\beta T, T) + 2(c_{\alpha\beta}p, T) \\  
&+ t_K\|\nabla p\|_0^2 + t_\theta\|\nabla T\|_0^2 \\  
=&  \|\bm{\varepsilon}(\bm{u})\|_0^2 + \left(\frac{1}{\lambda}\xi, \xi\right) - 2\left(\frac{\alpha}{\lambda}p, \xi\right) - 2\left(\frac{\beta}{\lambda}T, \xi\right)  
+ \left(\frac{\alpha^2}{\lambda}p, p\right) + \left(\frac{\beta^2}{\lambda}T, T\right) \\  
&+ 2\left(\frac{\alpha\beta}{\lambda}p, T\right) + (c_0p, p) + (a_0T, T) - 2(b_0p, T) + t_K\|\nabla p\|_0^2 + t_\theta\|\nabla T\|_0^2 \\  
=&  \|\bm{\varepsilon}(\bm{u})\|_0^2 + \left(-\frac{1}{\sqrt{\lambda}}\xi + \frac{\alpha}{\sqrt{\lambda}}p + \frac{\beta}{\sqrt{\lambda}}T,  
-\frac{1}{\sqrt{\lambda}}\xi + \frac{\alpha}{\sqrt{\lambda}}p + \frac{\beta}{\sqrt{\lambda}}T\right) \\  
&+ (c_0p, p) + (a_0T, T) - 2(b_0p, T) + t_K\|\nabla p\|_0^2 + t_\theta\|\nabla T\|_0^2 \\  
\geq & 0.  
\end{aligned}  
\]  
with equality holding only if \((\bm{u}, \xi, p, T) = (\bm{0}, 0, 0, 0)\). This demonstrates that \(\tilde{\mathcal{B}_1}\) is a symmetric positive definite operator, making \((\tilde{\mathcal{B}_1} \cdot, \cdot)\) an inner product. Consequently, the quantity  
\[
\|(\bm{u}, \xi, p, T)\|_{\mathcal{B}_1} = \Big( (\tilde{\mathcal{B}_1}(\bm{u}, \xi, p, T), (\bm{u}, \xi, p, T)) + \|\xi_0\|_0^2 \Big)^{\frac{1}{2}}
\]  
is a valid norm induced by this inner product. Furthermore, \(\mathcal{B}_1\) is a bounded linear operator, confirming that the expression in \eqref{3blocks norm} defines a proper norm.

The corresponding preconditioner for the continuous system associated with the norm \eqref{3blocks norm} can be expressed as:  
\begin{equation}\label{eq:precond3x3}  
\mathcal{B}_1^{-1}
=\begin{bmatrix}  
-  \text{div }\bm\varepsilon & 0 & 0 & 0 \\  
0 & \lambda^{-1}I + I_0 & -\frac{\alpha}{\lambda}I & -\frac{\beta}{\lambda}I \\  
0 & -\frac{\alpha}{\lambda}I & c_\alpha I - t_K \text{div}(\nabla) & c_{\alpha\beta}I \\  
0 & -\frac{\beta}{\lambda}I & c_{\alpha\beta}I & c_\beta I - t_\theta \text{div}(\nabla)  
\end{bmatrix}^{-1}.  
\end{equation}  

We first introduce a lemma that will be used in the following part of the paper to clarify the relationship between $\|\cdot\|_1$ and $\|\bm\varepsilon(\cdot)\|_0$.
\begin{lemma}\label{Korn's inequa}
There exists a constant $C_K$ such that the Korn’s inequality 
\begin{equation*}
               C_K \Vert  \bm  u\Vert_1^2   \le  \Vert\bm\varepsilon(\bm u)\Vert_0^2,
               \quad\forall \bm u\in [H^1_0(\Omega)]^d,  d=2,3.
\end{equation*}
holds true \cite{brenner2004korn}. 
\end{lemma}

\begin{lemma}[continuity]\label{le:A bound under B1}
Assume that the conditions (A1)-(A3) are satisfied. Let \(X = \bm{V} \times Q \times W \times W\) denote the Hilbert space equipped with the norm defined in \eqref{3blocks norm}. Let \(\mathcal{A}\) represent the operator defined in \eqref{TP matrix}. Then, there exists a constant \(C\), independent of any problem parameters, such that:
\begin{equation}\label{eq:A bound under B1}
\begin{aligned}  
    (\mathcal A(\bm u,\xi,p,T),(\bm v,\phi,q,S))
      \le& C \|(\bm u,\xi,p,T) \|_{\mathcal{B}_1} \|(\bm v,\phi ,q, S)\|_{\mathcal{B}_1}    
\end{aligned}  
\end{equation}
for all $(\bm u,\xi,p,T), (\bm v,\phi ,q, S)\in X$.

\end{lemma}
\begin{proof}
By the block form of $\mathcal{A}$ in \eqref{mathcal A blocks}, we have
\begin{equation} 
\begin{aligned}
   &(\mathcal A(\bm u,\xi,p,T),(\bm v,\phi,q,S))\\
   =&(\begin{bmatrix}
A_{0}     & B_0^*            \\             
B_0       & -C_0
\end{bmatrix}(\bm u,(\xi,p,T)), 
     (\bm v,(\phi, q, S)) ) \\
=&  (A_{0}\bm u,  \bm v)  +  (B_0\bm u,(\phi, q, S))   
  +(B_0^*(\xi,p,T),\bm v)
   +(-C_0(\xi,p,T),(\phi, q, S)).
\end{aligned}
\end{equation}
Given that \(\bm{u}\) and \(\bm{v}\) have vanishing traces, using Green's formula, the Cauchy-Schwarz inequality, and the definition of the mean-zero value parts in (\ref{def of phi_0}), we have
\begin{equation} 
\begin{aligned}
(A_{0} \bm u, \bm v)=( \bm\varepsilon(\bm u),\bm\varepsilon(\bm v))
  \le\|(\bm u,\xi,p,T) \|_{\mathcal{B}_1} \|(\bm v,\phi ,q, S)\|_{\mathcal{B}_1}  ,
\end{aligned}
\end{equation}
\begin{equation} 
\begin{aligned}
 (B_0\bm u,(\phi, q, S)) 
 =&-(\nabla\cdot\bm u, \phi)
 =-(\nabla\cdot\bm u, \phi_0)-(\nabla\cdot\bm u, \phi_m)\\
 \le &\|\nabla\cdot\bm u\|_0\|\phi_0\|_0
     -  \phi_m\int_\Omega\nabla\cdot\bm u\\
   \le &   \|\bm\varepsilon(\bm u)\|_0\|\phi_0\|_0
         -\phi_m\int_{\partial\Omega}\bm u\cdot\bm n\\
  \le &  \|(\bm u,\xi,p,T) \|_{\mathcal{B}_1} \|(\bm v,\phi ,q, S)\|_{\mathcal{B}_1}  ,
\end{aligned}
\end{equation}

\begin{equation} 
\begin{aligned}
 (B_0^*(\xi,p,T),\bm v)
 =&-(\xi,\nabla\cdot\bm v)
 =-(\xi_0,\nabla\cdot\bm v)-(\xi_m,\nabla\cdot\bm v)\\
 \le& -(\xi_0,\nabla\cdot\bm v)-\xi_m\int_\Omega\nabla\cdot\bm v  \\
 \le& -\|\xi_0\|_0\|\bm\varepsilon(\bm v)\|_0-\xi_m\int_{\partial\Omega}\bm v \cdot\bm n \\
  \le&      \|(\bm u,\xi,p,T) \|_{\mathcal{B}_1} \|(\bm v,\phi ,q, S)\|_{\mathcal{B}_1},
\end{aligned}
\end{equation}
and
\begin{equation} 
\begin{aligned}
 (-C_0(\xi,p,T),(\phi, q, S))
 =&((-\mathcal{B}_1+\mathcal{B}_{22})(\bm 0,\xi,p,T),(\bm 0,\phi, q, S))\\
 =&(-\tilde{\mathcal{B}_1}(\bm 0,\xi,p,T),(\bm 0,\phi, q, S)) \\
   \le & C \|(\bm u,\xi,p,T) \|_{\mathcal{B}_1} \|(\bm v,\phi ,q, S)\|_{\mathcal{B}_1} .
\end{aligned}
\end{equation}
Thus, we see that \eqref{eq:A bound under B1} holds true.
\end{proof}

\begin{theorem}[inf-sup condition]\label{infsup continuous 3blocks}
Let $X=\bm V\times Q\times W\times W$ be the Hilbert space with norm given in \eqref{3blocks norm}.
Assuming \textbf{assumptions}(A1)-(A3) hold, then there exists a constant $\zeta>0$, independent
of $\lambda, c_\alpha,c_{\alpha\beta},c_\beta,t_K$, and  $t_\theta$, such that the following inf-sup condition holds:
\begin{equation}\label{A continuous inf sup under B1}
 \inf_{(\bm u ,\xi  ,p , T )\in X}
  \sup_{(\bm v,\phi ,q, S) \in X}
              \frac{(\mathcal A(\bm u,\xi ,p, T),(\bm v,\phi ,q, S))}
         {\Vert(\bm u,\xi ,p, T)\Vert_{\mathcal{B}_1}\Vert(\bm v,\phi ,q, S)\Vert_{\mathcal{B}_1}}
         \ge \zeta >0.
\end{equation}
\end{theorem}
 \begin{proof}
 Let $(\bm u,\xi,p,T)\neq (\bm 0,0,0,0)$ be a given element of $X$. We aim to find a pair $(\bm v,\phi ,q, S)\in X$, such that
\begin{equation} 
 (\mathcal A(\bm u,\xi ,p, T),(\bm v,\phi ,q, S))
         \ge \zeta \Vert(\bm u,\xi ,p, T)\Vert_{\mathcal{B}_1}\Vert(\bm v,\phi ,q, S)\Vert_{\mathcal{B}_1}.
\end{equation}
For the given $\xi$, we write $\xi= P_m\xi+\xi_0$. By the theory for the Stokes problem (cf. Theorem 5.1 of \cite{girault2012finite}), 
there exists a function $\bm v_0\in \bm{V}$ and a constant $\eta_0$ depending only on the domain $\Omega$ such that
\begin{equation}\label{Stokes_infsup}
 (\nabla\cdot\bm v_0,\xi)=\Vert\xi_0\Vert_0^2 ~ \text{  and  }~
  (\bm\varepsilon(\bm v_0),\bm\varepsilon(\bm v_0))\le\eta_0\Vert\xi_0\Vert_0^2 .
\end{equation}
By Young's inequality, we can find a small positive $\eta_1$ which will be determined later 
such that $1-\eta_1 \eta_0>0$ and
\begin{equation}\label{1:(-bm v_0,0,0,0)}
\begin{aligned}
(\mathcal A(\bm u,\xi,p,T),(-\bm v_0,0,0,0))&=-(\bm\varepsilon(\bm u),\bm\varepsilon(\bm v_0))+( \nabla\cdot\bm v_0,\xi)\\ 
                         &\ge -\frac{1}{4\eta_1}(\bm\varepsilon(\bm u),\bm\varepsilon(\bm u))
                            - \eta_1(\bm\varepsilon(\bm v_0),\bm\varepsilon(\bm v_0))               
                      + \Vert\xi_0\Vert_0^2\\
                 &\ge -\frac{ 1 }{4\eta_1}\Vert\bm\varepsilon(\bm u)\Vert_0^2+ (1- \eta_1\eta_0)\Vert\xi_0\Vert_0^2.
\end{aligned}
\end{equation}
In the above estimates, we have used the properties in (\ref{Stokes_infsup}). 
We therefore have
\begin{equation}\label{1:(bm u,-xi,- p,- T)}
\begin{aligned}
&(\mathcal A(\bm u,\xi,p,T),(\bm u,-\xi,- p,- T))\\
   =&  (\bm\varepsilon(\bm u),\bm\varepsilon(\bm u))
      +\frac{1}{\lambda}\|\xi\|_0^2
      -2(\frac{\alpha}{\lambda}\xi,p)
      -2(\frac{\beta}{\lambda}\xi,T)
      +c_\alpha\|p\|_0^2  +  c_\beta\|T\|_0^2
      \\
     &+2(c_{\alpha\beta}p,T) +   t_K\Vert \nabla p\Vert_0^2      +t_\theta\Vert \nabla T\Vert_0^2  .
\end{aligned}
\end{equation} 
Take
$$(\bm v,\phi,q,S)= \frac{1}{2\eta_0}(-\bm v_0,0,0,0)+(\bm u,-\xi,-p,-T).$$
Let \(\eta_1 = \frac{1}{2\eta_0}\). By combining equations \eqref{1:(-bm v_0,0,0,0)} and \eqref{1:(bm u,-xi,- p,- T)}, we obtain:
\begin{equation} 
\begin{aligned}
       \|(\bm v,\phi ,q, S)\|_{\mathcal{B}_1} 
\le&    \|(-\frac{1}{2\eta_0}\bm v_0 ,0 ,0, 0)\|_{\mathcal{B}_1} 
         +\|(\bm u,-\xi,-p,-T)\|_{\mathcal{B}_1} 
           \\ 
\le& \sqrt{\frac{1}{4\eta_0^2}+1}  \| (\bm u ,\xi ,p, T)\|_{\mathcal{B}_1}
     \\     
\end{aligned}
\end{equation}
and
\begin{equation} 
\begin{aligned}
       &(\mathcal A(\bm u,\xi,p,T),(\bm v,\phi,q,S))\\
      =&
      (\mathcal A(\bm u,\xi,p,T),(\bm v_0,0,0,0))+(\mathcal A(\bm u,\xi,p,T),(\bm u,-\xi,-p,-T)) \\
 \ge& 
  \frac{1}{2\eta_0}\Big(-\frac{1}{4\eta_1}\Vert\bm\varepsilon(\bm u)\Vert_0^2
                   +(1-\eta_1\eta_0)\Vert\xi_0\Vert_0^2 \Big) 
+          \Big( (\bm\varepsilon(\bm u),\bm\varepsilon(\bm u))
      +\frac{1}{\lambda}\|\xi\|_0^2
      \\
      &
      -2(\frac{\alpha}{\lambda}\xi,p)
      -2(\frac{\beta}{\lambda}\xi,T)
      +c_\alpha\|p\|_0^2  +  c_\beta\|T\|_0^2
     +2(c_{\alpha\beta}p,T)
     +   t_K\Vert \nabla p\Vert_0^2      \\
     &+t_\theta\Vert \nabla T\Vert_0^2 \Big)\\
  \ge&  \frac14\Vert \bm\varepsilon(\bm u)\Vert_0^2 
      +\frac{1}{\lambda}\Vert\xi\Vert_0^2
       +\frac{1}{4\eta_0}\Vert\xi_0\Vert_0^2 
       -2(\frac{\alpha}{\lambda}\xi,p)-2(\frac{\beta}{\lambda}\xi,T)
       +c_\alpha\|p\|_0^2  \\
      & 
      +  c_\beta\|T\|_0^2
     +2(c_{\alpha\beta}p,T)
     +   t_K\Vert \nabla p\Vert_0^2      +t_\theta\Vert \nabla T\Vert_0^2  \\
 \ge&\min\{\frac14, \frac{1}{4\eta_0}, 1\} \|(\bm u,\xi,p,T)\|_{\mathcal{B}_1}^2
     =\min\{\frac14, \frac{1}{4\eta_0}\} \|(\bm u,\xi,p,T)\|_{\mathcal{B}_1}^2\\
  \ge& \zeta \|(\bm u,\xi,p,T)\|_{\mathcal{B}_1}\|(\bm v,\phi,q,S)\|_{\mathcal{B}_1},
\end{aligned}
\end{equation}   
  where $\zeta=\min\{\frac14, \frac{1}{4\eta_0}\}/ \sqrt{\frac{1}{4\eta_0^2}+1}$ is a  constant  independent of $\lambda, \alpha, \beta,  c_\alpha, 
 c_{\alpha\beta}, c_\beta$, $t_K$ and $t_\theta$. 
Thus, we obtain the inf-sup condition \eqref{A continuous inf sup under B1}.
\end{proof}
\begin{remark}
Based on the norm validity test and the stability proof in Theorem \ref{infsup continuous 3blocks}, it is evident that the norm definition remains valid, and the conclusions of the theorem continue to hold when \(c_0 = b_0\) or \(a_0 = b_0\). Consequently, the preconditioner \(\mathcal{B}_1\) is robust with respect to all physical parameters, even in the limiting case where the coefficients \(a_0\), \(b_0\), and \(c_0\) approach zero \cite{brun2020monolithic}.
\end{remark}

\subsection{A block diagonal parameter-robust preconditioner}

In this subsection, we aim to construct a block diagonal preconditioner. As noted in \cite{lee2017parameter,hong2021new}, the assumption \(c_0 \sim \frac{\alpha^2}{\lambda}\) was introduced in the development of a block diagonal preconditioner for the quasi-static Biot model. In this work, we extend this assumption from the Biot model to the thermo-poroelasticity model. Consequently, in addition to the \textbf{assumptions} (A1)-(A3), we present the following parameter-related assumptions for the remainder of this section:
\begin{equation}\label{assumptions:Cp CT }
  C_p = c_0 - b_0 \geq \frac{\alpha^2}{\lambda},\quad C_T = a_0 - b_0 \geq \frac{\beta^2}{\lambda}.
\end{equation}



Using a similar methodology as in the previous analysis, we define the second type of parameter-dependent norm as follows:
\begin{equation}\label{norm B2}
\begin{aligned}
 \Vert(\bm u,\xi ,p, T)\Vert_{\mathcal{B}_2}^2
                 = & \Vert \bm\varepsilon(\bm u)\Vert_0^2 
                 +\frac1\lambda\Vert\xi\Vert_0^2
                 +\Vert\xi_0\Vert_0^2  
                 +C_p\Vert p\Vert_0^2   +   t_K\Vert\nabla p\Vert_0^2\\
                 &+C_T\Vert T\Vert_0^2   +   t_\theta\Vert\nabla T\Vert_0^2
                       ,    
\end{aligned}
\end{equation}
where $\xi_0$ is again the mean-value part of $\xi$.
Using Green’s formula and the boundary conditions (\ref{BC}), we observe that:
\begin{equation} 
\begin{aligned}
  \|(\bm u,\xi,p,T)\|_{\mathcal{B}_2}^2 
 = ( {\mathcal{B}_2}(\bm u,\xi, p, T), 
          (\bm u,\xi, p, T)
   ),
 \end{aligned}
\end{equation}  
where
\begin{equation}\label{operator B2}
\mathcal{B}_2=
\begin{bmatrix}
-\text{div }\bm\varepsilon    & 0             &  0                  &  0               \\             
0                & \lambda^{-1}I+I_0    &0         &0\\
0                          &0     &C_p  I-t_K\text{div}( \nabla)&0\\
0                          & 0        &0     &C_T  I-t_\theta\text{div}(\nabla)\\
\end{bmatrix}.
\end{equation}  
Similar to \(\mathcal{B}_1\), the operator \(\mathcal{B}_2\) is subject to the same boundary conditions as those in the original problem \eqref{general TP_Model}. It is easy to see that the parameter-dependent norm defined in \eqref{norm B2} motivates  \(\mathcal{B}_2\). Its inverse
\begin{equation}\label{eq:preconddiag}
\begin{aligned}
\mathcal{B}_2^{-1}
=
\begin{bmatrix}
-\text{div }\bm\varepsilon    & 0             &  0                  &  0               \\             
0                & \lambda^{-1}I+I_0   &0         &0\\
0                          &0     &C_p  I-t_K\text{div}( \nabla)   &0\\
0                          & 0        &0     & C_T  I-t_\theta\text{div}(\nabla) \\
\end{bmatrix}^{-1}.
\end{aligned}
\end{equation}  
The first, third, and fourth blocks of this block-diagonal operator correspond to the inverses of standard second-order elliptic operators, for which well-established preconditioners are available to replace the exact inverses in the discrete case. The process for the second block 
follows the same approach as described in the previous subsection.


\begin{lemma}[continuity]\label{le:A bound under B2}
Assume that the conditions (A1)–(A3) are satisfied. Let \( X = \bm{V} \times Q \times W \times W \) denote the Hilbert space equipped with the norm defined in \eqref{3blocks norm}. Let \(\mathcal{A}\) be the operator defined in \eqref{TP matrix}. Then, there exists a constant \( C \), independent of the model parameters, such that:  
\begin{equation}\label{eq:A bound under B2}
\begin{aligned}  
   (\mathcal A(\bm u,\xi,p,T),(\bm v,\phi,q,S)) 
      \le& C \|(\bm u,\xi,p,T) \|_{\mathcal{B}_2} \|(\bm v,\phi ,q, S)\|_{\mathcal{B}_2}   
\end{aligned}        
\end{equation}
for all $(\bm u,\xi,p,T)\in X, (\bm v,\phi ,q, S)\in X$.  
\end{lemma}

\begin{proof}
Similar to the proof of \Cref{le:A bound under B1}, we have
\begin{equation} 
\begin{aligned}
   &(\mathcal A(\bm u,\xi,p,T),(\bm v,\phi,q,S))\\
=&  (A_{0}\bm u,  \bm v)  +  (B_0\bm u,(\phi, q, S))   
  +(B_0^*(\xi,p,T),\bm v)
   +(-C_0(\xi,p,T),(\phi, q, S)),
\end{aligned}
\end{equation}
where
\begin{equation} 
\begin{aligned}
(A_{0}\bm u,  \bm v)=(\bm\varepsilon(\bm u),\bm\varepsilon(\bm v))
  \le\|(\bm u,\xi,p,T) \|_{\mathcal{B}_2} \|(\bm v,\phi ,q, S)\|_{\mathcal{B}_2}  ,
\end{aligned}
\end{equation}
\begin{equation} 
\begin{aligned}
 (B_0\bm u,(\phi, q, S)) 
 =-(\nabla\cdot\bm u, \phi)
  \le\|(\bm u,\xi,p,T) \|_{\mathcal{B}_2} \|(\bm v,\phi ,q, S)\|_{\mathcal{B}_2}  ,
\end{aligned}
\end{equation}
\begin{equation} 
\begin{aligned}
 (B_0^*(\xi,p,T),\bm v)
 =-(\xi,\nabla\cdot\bm v)
  \le      \|(\bm u,\xi,p,T) \|_{\mathcal{B}_2} \|(\bm v,\phi ,q, S)\|_{\mathcal{B}_2},
\end{aligned}
\end{equation}
and
\begin{equation} 
\begin{aligned}
 (-C_0(\xi,p,T),(\phi, q, S))
  \le\|(\bm u,\xi,p,T) \|_{\mathcal{B}_2} \|(\bm v,\phi ,q, S)\|_{\mathcal{B}_2}  .
\end{aligned}
\end{equation}
Thus, we obtain the estimate \eqref{eq:A bound under B2}.
\end{proof}

Hence, \(\mathcal{A}\) is a linear bounded operator with respect to the norm \(\mathcal{B}_2\). Next, we will show that \(\mathcal{A}\) satisfies the inf-sup condition under the norm \(\mathcal{B}_2\).

\begin{theorem}[inf-sup condition]\label{infsup continuous diagonal}
Let $X=\bm{V}\times Q\times W\times W$ be the Hilbert space with the norm given 
in \eqref{norm B2}. Assuming \textbf{assumptions} (A1)-(A3) and \eqref{assumptions:Cp CT } hold true,
 then there exists a constant $\zeta>0$, independent
of $\lambda, c_\alpha,c_{\alpha\beta},c_\beta,t_K$, and  $t_\theta$, such that the following inf-sup condition holds:
\begin{equation}\label{A continuous inf sup under B2}
 \inf_{(\bm u ,\xi  ,p , T )\in X}
  \sup_{(\bm v,\phi ,q, S) \in X}
              \frac{(\mathcal A(\bm u,\xi ,p, T),(\bm v,\phi ,q, S))}
         {\Vert(\bm u,\xi ,p, T)\Vert_{\mathcal{B}_2}\Vert(\bm v,\phi ,q, S)\Vert_{\mathcal{B}_2}}
         \ge \zeta >0.
\end{equation}
\end{theorem}
 \begin{proof}
 Let $(\bm u,\xi,p,T)\neq (\bm 0,0,0,0)$ be 
 a given element of $X$.
 We aim to find a pair $(\bm v,\phi ,q, S)\in X$,  such that
\begin{equation}\label{A>uxipt vphiqs}
 (\mathcal A(\bm u,\xi ,p, T),(\bm v,\phi ,q, S))
         \ge \zeta \Vert(\bm u,\xi ,p, T)\Vert_{\mathcal{B}_2}\Vert(\bm v,\phi ,q, S)\Vert_{\mathcal{B}_2}.
\end{equation}

Firstly, by using (\ref{Stokes_infsup}) and (\ref{1:(-bm v_0,0,0,0)}), there exists a $\bm v_0\in \bm V$ 
and two constants $\eta_0$ and $\eta_1$ such that 
\begin{equation}\label{(v_0,0,0,0)}
\begin{aligned}
(\mathcal A(\bm u,\xi,p,T),(-\bm v_0,0,0,0))
                 &\ge -\frac{  1 }{4\eta_1}\|\bm\varepsilon(\bm u)\|_0^2 + (1- \eta_1\eta_0)\Vert\xi_0\Vert_0^2.
\end{aligned}
\end{equation}
Secondly, by Young's inequality, we have
\begin{equation}
\begin{aligned}\label{(u,-xi,0,0)}
&(\mathcal A(\bm u,\xi,p,T),(\bm u,-\xi,0,0))\\
   =&  \Vert\bm\varepsilon(\bm u)\Vert_0^2+\lambda^{-1}\Vert\xi\Vert_0^2-(\frac\alpha\lambda p,\xi)-(\frac{\beta}{\lambda}T,\xi) \\
   \ge&\Vert\bm\varepsilon(\bm u)\Vert_0^2+\lambda^{-1}\Vert\xi\Vert_0^2
       -(\frac{\alpha^2}{\lambda}p,p) - \frac{1}{4}(\lambda^{-1}\xi,\xi) 
       -(\frac{\beta^2}{\lambda}T,T)- \frac14(\lambda^{-1}\xi,\xi) \\
     \ge&  \Vert\bm\varepsilon(\bm u)\Vert_0^2+\frac12\lambda^{-1}\Vert\xi\Vert_0^2
       -(\frac{\alpha^2}{\lambda}p,p) 
       -(\frac{\beta^2}{\lambda}T,T) .      \\
\end{aligned}
\end{equation}
Thirdly, by using the definition of $c_\alpha,c_\beta$ and $c_{\alpha\beta}$ and equation \eqref{1:(bm u,-xi,- p,- T)}, we have
\begin{equation}\label{Cmax1}
\begin{aligned}
& (\mathcal A(\bm u,\xi,p,T),(\bm u,-\xi,- p,- T))\\
   =&  (\bm\varepsilon(\bm u),\bm\varepsilon(\bm u))
      +\frac{1}{\lambda}\|\xi\|_0^2
      -2(\frac{\alpha}{\lambda}\xi,p)
      -2(\frac{\beta}{\lambda}\xi,T)
      +c_\alpha\|p\|_0^2  +  c_\beta\|T\|_0^2
     \\
     & +2(c_{\alpha\beta}p,T) +   t_K\Vert \nabla p\Vert_0^2      +t_\theta\Vert \nabla T\Vert_0^2 \\
   =& \|\bm\varepsilon(\bm u)\|_0^2 
    +(-\frac{1}{\sqrt{\lambda}}\xi + \frac{\alpha}{\sqrt{\lambda}}p + \frac{\beta}{\sqrt{\lambda}}T,  
  -\frac{1}{\sqrt{\lambda}}\xi + \frac{\alpha}{\sqrt{\lambda}}p
  + \frac{\beta}{\sqrt{\lambda}}T)
  \\ 
   &+ (c_0p, p) + (a_0T, T) - 2(b_0p, T)
   +   t_K\Vert \nabla p\Vert_0^2      +t_\theta\Vert \nabla T\Vert_0^2 \\ 
   \ge& \|\bm\varepsilon(\bm u)\|_0^2 
   + C_p\| p\|_0^2  +  C_T\| T\|_0^2  
       +   t_K\Vert \nabla p\Vert_0^2      +t_\theta\Vert \nabla T\Vert_0^2.
\end{aligned}
\end{equation} 
Finally, let 
\begin{equation}
\begin{aligned}
(\bm v,\phi,q, S)=& \eta_2(- \bm v_0,0,0,0) +\frac{1}{2}(\bm u, -\xi, 0, 0)
+ (\bm u,-\xi,-p,-T),
\end{aligned}
\end{equation}
where $\eta_2$ is a constant that we will select later. Combining equation \eqref{(v_0,0,0,0)}, equation \eqref{(u,-xi,0,0)}, and \eqref{Cmax1} above, we obtain
\begin{equation}\label{lower boundedness}
\begin{aligned}
& (\mathcal A(\bm u,\xi,p,T),(\bm v,\phi,q,S))\\
  \ge& \eta_2\Big(-\frac{1}{4\eta_1}  \|\bm\varepsilon(\bm u)\|_0^2 
                   +(1-\eta_1\eta_0)\Vert\xi_0\Vert_0^2 \Big) 
  +   \frac12\Big(  \|\bm\varepsilon(\bm u)\|_0^2 
                   +\frac12\lambda^{-1}\Vert\xi\Vert_0^2 
                     \\
       & - (\frac{\alpha^2}{\lambda}p,p) -(\frac{\beta^2}{\lambda}T,T)                 \Big)  
+          \Big(  \|\bm\varepsilon(\bm u)\|_0^2 
   + C_p\| p\|_0^2  +  C_T\| T\|_0^2  
       +   t_K\Vert \nabla p\Vert_0^2    \\
       &+t_\theta\Vert \nabla T\Vert_0^2  \Big) \\ 
        \ge&  \big(\frac32 -\frac{\eta_2}{4\eta_1} \big)\|\bm\varepsilon(\bm u)\|_0^2 
      +\frac14\lambda^{-1}\Vert\xi\Vert_0^2
       +\eta_2(1-\eta_1\eta_0)\Vert\xi_0\Vert_0^2
       + \big(C_p-\frac12\frac{\alpha^2}{\lambda}\big)\Vert p\Vert_0^2 
     \\
      &
     + \big(C_T-\frac12\frac{\beta^2}{\lambda}\big) \Vert T\Vert_0^2+t_K\Vert\nabla p\Vert_0^2+t_\theta\Vert\nabla T\Vert_0^2
     \\
      \ge&  \big(\frac32-\frac{\eta_2}{4\eta_1}\big)\|\bm\varepsilon(\bm u)\|_0^2 
      +\frac14\lambda^{-1}\Vert\xi\Vert_0^2
       +\eta_2(1-\eta_1\eta_0)\Vert\xi_0\Vert_0^2
      + \frac12C_p \Vert p\Vert_0^2 
     \\
     &+ \frac12 C_T \Vert T\Vert_0^2
      +t_K\Vert\nabla p\Vert_0^2+t_\theta\Vert\nabla T\Vert_0^2
       \\
  \ge&\zeta_1\Big ( \|\bm\varepsilon(\bm u)\|_0^2  
            +\frac1\lambda\Vert\xi\Vert_0^2+\Vert\xi_0\Vert_0^2
            +C_p \Vert p\Vert_0^2
            +C_T \Vert T\Vert_0^2  
       + t_K\Vert\nabla p\Vert_0^2 \\
      & +t_\theta\Vert\nabla T\Vert_0^2\Big  ),\\
\end{aligned}
\end{equation}
where 
$\zeta_1= \min\{ \frac32-\frac{\eta_2}{4\eta_1},\frac{1}{4},\eta_2(1-\eta_1\eta_0),\frac{1}{2},1\}$.
Note that in the above inequality, by the assumptions in
 \eqref{assumptions:Cp CT }, 
there holds
$$
C_p -\frac12\frac{\alpha^2}{\lambda}\ge \frac12C_p ~\text{ and}~
C_T  -\frac12\frac{\beta^2}{\lambda}\ge \frac12C_T .
$$ 
If we take $\eta_2=\eta_1=\frac{1}{2\eta_0}$, then  $\zeta_1=\min\{\frac{1}{4},\frac{1}{4\eta_0} \}>0$ is a  constant  independent of $\lambda, \alpha, \beta,  c_\alpha, 
 c_{\alpha\beta}, c_\beta, t_K$ and $t_\theta$.
We have
\begin{equation}\label{upper boundedness}
\begin{aligned}
       &\|(\bm v,\phi ,q, S)\|_{\mathcal{B}_2}\\
\le&  \big(  \|(-\eta_2\bm v_0 ,0 ,0, 0)\|_{\mathcal{B}_2}
         +\|(\frac12\bm u,-\frac12\xi,0,0)\|_{\mathcal{B}_2}
         +\|(\bm u,-\xi,-p,-T)\|_{\mathcal{B}_2} 
       \big)^{\frac12}    \\
\le& \zeta_2\| (\bm u ,\xi ,p, T)\|_{\mathcal{B}_2}
     \\     
\end{aligned}
\end{equation}
where $\zeta_2=\sqrt{\eta_2^2+\frac54}=\sqrt{\frac{1}{4\eta_0^2}+\frac54}$ is a positive constant independent of $\lambda$, $\alpha$, $\beta$, $c_\alpha$, 
 $c_{\alpha\beta}$, $c_\beta$, $t_K$ and $t_\theta$.
 
Thus, combing \eqref{lower boundedness} and \eqref{upper boundedness}, for any given $(\bm u,\xi, p, T)$, we can find a pair $(\bm v,\phi,q,S)$ such that 
\begin{equation}
\begin{aligned}
    &(\mathcal A(\bm u,\xi,p,T),(\bm v,\phi,q,S))\\
\ge& \zeta_1\Big ( \Vert\bm\varepsilon(\bm u)\Vert_1^2 
            +\frac1\lambda\Vert\xi\Vert_0^2+\Vert\xi_0\Vert_0^2
            +C_p \Vert p\Vert_0^2
            +C_T \Vert T\Vert_0^2  
       + t_K\Vert\nabla p\Vert_0^2+t_\theta\Vert\nabla T\Vert_0^2   \Big  )\\
 =& \zeta_1  \|(\bm u ,\xi ,p, T)\|_{\mathcal{B}_2}^2  \\ 
 \ge& \frac{\zeta_1}{\zeta_2}  \|(\bm u ,\xi ,p, T)\|_{\mathcal{B}_2}
                           \|(\bm v,\phi ,q, S)\|_{\mathcal{B}_2}  \\      
 \ge&\zeta  \|(\bm u ,\xi ,p, T)\|_{\mathcal{B}_2}
                           \|(\bm v,\phi ,q, S)\|_{\mathcal{B}_2} ,\\                               
\end{aligned}
\end{equation}
is satisfied with \(\zeta = \frac{\zeta_1}{\zeta_2}\), which consequently ensures that \eqref{A>uxipt vphiqs} is also valid.
\end{proof}

\section{Discretization and construction of preconditioners}\label{sec:dis}
In this section, we present finite element discretizations for the four-field formulation discussed earlier
and demonstrate that it is possible to identify parameter-robust preconditioners for the discretized problems. 
For the discretized spaces \(\bm{V}_h \times Q_h\), the classical Stokes inf-sup stability condition must be satisfied, i.e.,
\begin{equation}\label{infsup for dis}
\begin{aligned} 
 \inf_{\xi\in Q_h\cap L_0^2}\sup_{\bm v\in \bm{V}_h}\frac{(\nabla\cdot\bm v,\xi)}{\|\bm v\|_1\|\xi\|_0}\ge \gamma>0
\end{aligned}
\end{equation}
where $\gamma$ is independent of mesh size $h$, which means $\bm V_h\times Q_h$ is a stable Stokes pair.
For a given mesh size \(h > 0\), let \(\mathcal{T}_h\) denote a tesselation of \(\Omega\) into triangular elements. We assume that the triangulation is shape-regular and quasi-uniform. The following finite element spaces are then chosen: 
\begin{equation}\label{discreteFESpace}
\begin{aligned} 
 &\bm{V}_h=\{  \bm{v}_h \in [H_0^1(\Omega)]^d \cap {\bm{C} }^0(\bar{\Omega});  \bm{v}_h|_E\in \mathbb P_2(E),\forall E\in\mathcal T_h    \} , \\
 &    Q_h=\{     \phi_h \in L^2(\Omega) \cap C^0(\bar{\Omega}); \phi_h|_E\in \mathbb P_1(E),\forall E\in\mathcal T_h    \} , \\
 &    W_h=\{      q_h\in  H_0^1(\Omega) \cap C^0(\bar{\Omega}); q_h|_E\in \mathbb P_2(E),\forall E\in\mathcal T_h    \}. \\
\end{aligned}
\end{equation}
Other stable Stokes pairs, such as the MINI element, can also be used, and higher-order elements \cite{oyarzua2016locking} may be chosen for \(W_h\). Given sufficient regularities of the solution, we will have the corresponding approximation properties \cite{oyarzua2016locking} of the subspaces specified above. 

Then the discrete counterpart of  \eqref{TP_Model_static} reads as:
find $(\bm u_h, \xi_h , p_h, T_h)\in\bm{V}_h\times Q_h\times W_h\times W_h $ 
such that for any $(\bm{v}_h,\phi_h,q_h,S_h)\in \bm{V}_h\times Q_h\times W_h\times W_h $,
\begin{equation}\label{TP_var_dis}
\begin{aligned} 
(\bm\varepsilon(\bm u_h),\bm\varepsilon(\bm{v}_h)) - (\nabla\cdot\bm{v}_h,\xi_h)&=(\bm f,\bm{v}_h)       ,   \\
  -(\nabla\cdot\bm u_h,\phi_h)-({\lambda^{-1}}\xi_h,\phi_h)+(\lambda^{-1}p_h,\phi_h)+(\frac{\beta}{\lambda}T_h,\phi_h)&=0 ,\\
 (\lambda^{-1}\xi_h ,q_h)-(c_\alpha p_h ,q_h)- (t_K\nabla p_h,\nabla q_h)- (c_{\alpha\beta}T_h ,q_h)
                &=(-\tilde g,q_h),\\
 (\frac{\beta}{\lambda}\xi_h ,S_h)-(c_{\alpha\beta}p_h ,S_h)-(c_\beta T_h ,S_h)   
  -  (t_\theta\nabla T_h,\nabla S_h)&=(-\tilde {H},S_h).\\
\end{aligned}
\end{equation}

Next, we present two stability theorems for the discrete problem.
\begin{theorem}\label{3blocks norm dis}
Suppose that $\bm{V}_h\subset \bm V, Q_h\subset Q, W_h\subset W$ are finite element spaces. Assume that the pair
$\bm V_h\times Q_h$ satisfies the inf-sup condition \eqref{infsup for dis}. 
Let $X_h=\bm V_h\times Q_h\times W_h\times W_h$ be the Hilbert space with norm given in \eqref{3blocks norm} and $\mathcal A_h: X_h \to X_h^*$ the corresponding discrete operator given by \eqref{TP_Model_var}. Then there is a constant $\zeta>0$ independent
of $\alpha, \beta, \lambda, c_\alpha, c_{\alpha\beta}, c_\beta, t_K$, and $t_\theta$ satisfying (A1)-(A3), as well as the mesh size $h$ 
 such that 
 \begin{equation} 
 \inf_{(\bm u ,\xi  ,p , T )\in X_h}\sup_{(\bm v,\phi ,q, S)\in X_h}
              \frac{(\mathcal A_h(\bm u,\xi ,p, T),(\bm v,\phi ,q, S))}
         {\Vert(\bm u,\xi ,p, T)\Vert_{\mathcal{B}_1}\Vert(\bm v,\phi ,q, S)\Vert_{\mathcal{B}_1}}
         \ge \zeta >0
\end{equation}
 holds.
\end{theorem}

\begin{theorem}\label{diagonal norm dis}
Suppose that $\bm{V}_h\subset \bm V, Q_h\subset Q, W_h\subset W$ are finite element spaces. Assume that the pair
$\bm V_h\times Q_h$ satisfies the inf-sup condition \eqref{infsup for dis}. 
 Let $X_h=\bm{V}_h\times Q_h\times W_h\times W_h$ be the Hilbert space with norm given in \eqref{norm B2} and $\mathcal A_h: X_h \to X_h^*$ the corresponding discrete operator given by \eqref{TP_Model_var}. Then there is a constant  $\zeta>0$
   independent
of $\alpha, \beta, \lambda, c_\alpha, c_{\alpha\beta}, c_\beta, t_K$, and $t_\theta$ satisfying (A1)-(A3) and \eqref{assumptions:Cp CT }, as well as the mesh size $h$ 
 such that 
 \begin{equation} 
 \inf_{(\bm u ,\xi  ,p , T )\in X_h}\sup_{(\bm v,\phi ,q, S)\in X_h}
              \frac{(\mathcal A_h(\bm u,\xi ,p, T),(\bm v,\phi ,q, S))}
         {\Vert(\bm u,\xi ,p, T)\Vert_{\mathcal{B}_2}\Vert(\bm v,\phi ,q, S)\Vert_{\mathcal{B}_2}}
         \ge \zeta >0
\end{equation}
holds.  
\end{theorem}
The proofs of Theorem \ref{3blocks norm dis} and Theorem \ref{diagonal norm dis} are analogous to the proof of Theorem \ref{infsup continuous 3blocks} and Theorem \ref{infsup continuous diagonal}. We now only give a short proof.
 
\begin{proof}For any \((\bm u_h, \xi_h, p_h, T_h) \in X_h\), by the inf-sup condition in \eqref{infsup for dis}, Korn's inequality, and the Stokes problem theory, there exists a constant \(\eta_0\), depending only on the domain \(\Omega\), and \(\bm{v}_{0,h} \in \bm{V}_h\) such that

\[
(\text{div} \, \bm{v}_{0,h}, \xi_h) = \|\xi_{h,0}\|_0 \quad \text{and} \quad (\bm{v}_{0,h}, \bm{v}_{0,h}) \le \eta_0 \|\xi_{h,0}\|_0,
\]
where \(\xi_{h,0} = \xi_h - P_m \xi_h\) as defined in \eqref{def of phi_0}. This implies that for a sufficiently small constant \(\eta_1\), depending only on the domain \(\Omega\), we have

\[
\begin{aligned}
(\mathcal A_h (\bm u_h, \xi_h, p_h, T_h),& (-\bm{v}_{0,h}, 0, 0, 0)) \\
&= -(\bm \varepsilon(\bm u_h), \bm \varepsilon(\bm{v}_{0,h})) + (\nabla \cdot \bm{v}_{0,h}, \xi_h) \\
&\ge -\frac{1}{4\eta_1} (\bm \varepsilon(\bm u_h), \bm \varepsilon(\bm u_h)) - \eta_1 (\bm \varepsilon(\bm{v}_{0,h}), \bm \varepsilon(\bm{v}_{0,h})) + \|\xi_{0,h}\|_0^2 \\
&\ge -\frac{C_K}{4\eta_1} \|\bm u_h\|_1^2 + (1 - \eta_1 \eta_0) \|\xi_{h,0}\|_0^2.
\end{aligned}
\]
The remaining steps follow similarly to those in Theorems \ref{3blocks norm dis} and \ref{diagonal norm dis}. 
\end{proof}

\section{Numerical experiments}\label{sec:Numerical results}

In this section, we provide numerical experiments to demonstrate the computational accuracy and efficiency of the algorithms proposed in the previous sections. In all of the examples, we let $\Omega=[0,1]^2$ and assume Dirichlet boundary
conditions for displacement, pressure, and temperature on all of $\partial\Omega$.
We consider triangular partitions $\mathcal T_h$ and finite element spaces \eqref{discreteFESpace}.
The examples were implemented using the finite element library Firedrake \cite{firedrake}.

First, we verify the convergence properties of the discretization.

\begin{example}[Error convergence]\label{ex:example1}
  We consider the time-dependent problem \eqref{TP_Model} with the material
  parameters set as $\mu=0.5$, $\lambda=3$, $\alpha=3$, $\bm K=K\bm I$, $K=1$, $c_0=0.3$,
  $\beta=2$, $\bm\Theta = \theta\bm I$, $\theta=2$, $a_0=4$, $b_0=0.1$. Then, the body force $\bm f$,
  the mass source $g$, and the heat source $h$ are chosen such that the exact
  solution is given by
\begin{equation}\label{Example1}  
\begin{aligned}   
      & \bm u(x,t)= \big(   e^{-t}\sin(\pi x)\sin(\pi y) ,e^{-t}\sin(\pi x)\sin(\pi y)   \big)^T,              \\
      & p(x,t) =  \lambda e^{-t}\sin(\pi x) \cos(\pi y)    ,      \\
      & T(x,t) =  \mu e^{-t}\cos(\pi x) \sin(\pi y)   .
\end{aligned}
\end{equation} 
Setting $\Delta t=10^{-2}$ in the backward Euler scheme \eqref{TP_Model_static},
we simulate the system's evolution until $t_f=10^{-1}$. Using their respective natural
norms the errors at the final time of
$\bm u_h$, $\xi_h$, $p_h$, $T_h$ are shown in \Cref{tab:example1}. Here, optimal rates
for all the variables can be observed.

\begin{table}
  \centering
  \footnotesize{
    \begin{tabular}{c|cccc}
      \hline
      $h$ & $\lVert \bm{u}-\bm{u}_h\rVert_{1}$ & $\lVert \xi-\xi_h\rVert_{0}$ & $\lVert p-p_h\rVert_{1}$ & $\lVert T-T_h\rVert_{1}$\\
      \hline
0.707107   &4.74963(--)     &2.54616(--)    &0.910795(--)     &10.3773(--)\\        
0.353553   &2.67629(0.83)   &1.26063(1.01)  &0.739976(0.30)   &5.40733(0.94)\\
0.176777   &0.59638(2.17)   &0.31858(1.98)  &0.190834(1.95)   &1.69665(1.67)\\ 
0.0883883  &0.07901(2.92)   &0.06618(2.27)  &0.048254(1.98)   &0.45525(1.89)\\  
0.0441942  &0.01111(2.83)   &0.01561(2.08)  &0.012094(2.00)   &0.11606(1.97)\\ 
0.0220971  &0.00183(2.60)   &0.00384(2.02)  &0.003025(2.00)   &0.02916(2.00)\\  
0.0110485  &0.00037(2.31)   &0.00095(2.01)  &0.000756(2.00)   &0.00730(2.00)\\  
0.00552426 &0.00008(2.12)   &0.00023(2.00)  &0.000189(2.00)   &0.00182(2.00)\\
\hline
    \end{tabular}
  }
  \caption{
    Approximation errors at $t=t_f$ of the backward Euler scheme \eqref{TP_Model_static}
    for thermo-poroelastic model setup in \Cref{ex:example1}. Discretization
    by \eqref{discreteFESpace} finite element spaces. The estimated rate of convergence is shown in the
    brackets.
  }
  \label{tab:example1}
\end{table}

\end{example}

Next, we demonstrate parameter robustness of the two proposed preconditioners. With
both $\mathcal{B}^{-1}_1$ in \eqref{eq:precond3x3} and $\mathcal{B}^{-1}_2$ in  \eqref{eq:preconddiag} we discuss the performance of the exact
preconditioners (with the elliptic operators inverted by LU decomposition) and inexact preconditioners
where the preconditioner blocks are approximated by algebraic multigrid (AMG).
Specifically, we run a single V-cycle of BoomerAMG solver from Hypre\cite{hypre} using the
library's default settings. The treatment of the dense block due to $L^2_0$-projection is described further in \Cref{rmrk:dense}. With
\eqref{eq:precond3x3} we apply AMG to the $3\times 3$ block inducing the inner
product on $Q_h\times W_h\times W_h$. This operator's challenges for AMG
are briefly discussed in \Cref{rmrk:amg}.

\begin{example}[Preconditioning]\label{ex:precon}
We investigate the performance of the preconditioners in terms of the stability
of preconditioned MinRes iterations under parameter variations and mesh refinement.
In the following, the iterations are always started from a 0 initial guess and terminate once the relative preconditioned residual norm is below $10^{-12}$. Setting $\Delta t = 10^{-2}$, the MinRes solver is applied to the linear system due to \eqref{TP_Model_static} and discretization \eqref{discreteFESpace}.

Using setup of \Cref{ex:example1} we first consider robustness under variations of $\beta$, $\lambda$, $\theta$ and $K$ while fixing $\mu=0.5$, $c_0=a_0=2$, $b_0=\alpha=1$.
As illustrated in \Cref{fig:robust}, both preconditioners, in their exact and inexact variants, result in bounded iteration counts. The tabular representation of data from \Cref{fig:robust} is provided in Supplementary results in \Cref{sec:extras}. From both \Cref{fig:robust} and \Cref{tab:robust_amg,tab:robust_lu}, it can be seen that both preconditioners are robust with respect to physical parameter variations and mesh refinements. Generally, the preconditioner based on \(\mathcal{B}_1\) achieves faster convergence compared to the block diagonal preconditioner based on \(\mathcal{B}_2\). We also observe that the inexact versions require up to twice as many iterations as their exact counterparts.
  
  \begin{figure}
    \centering
\includegraphics[width=0.495\textwidth]{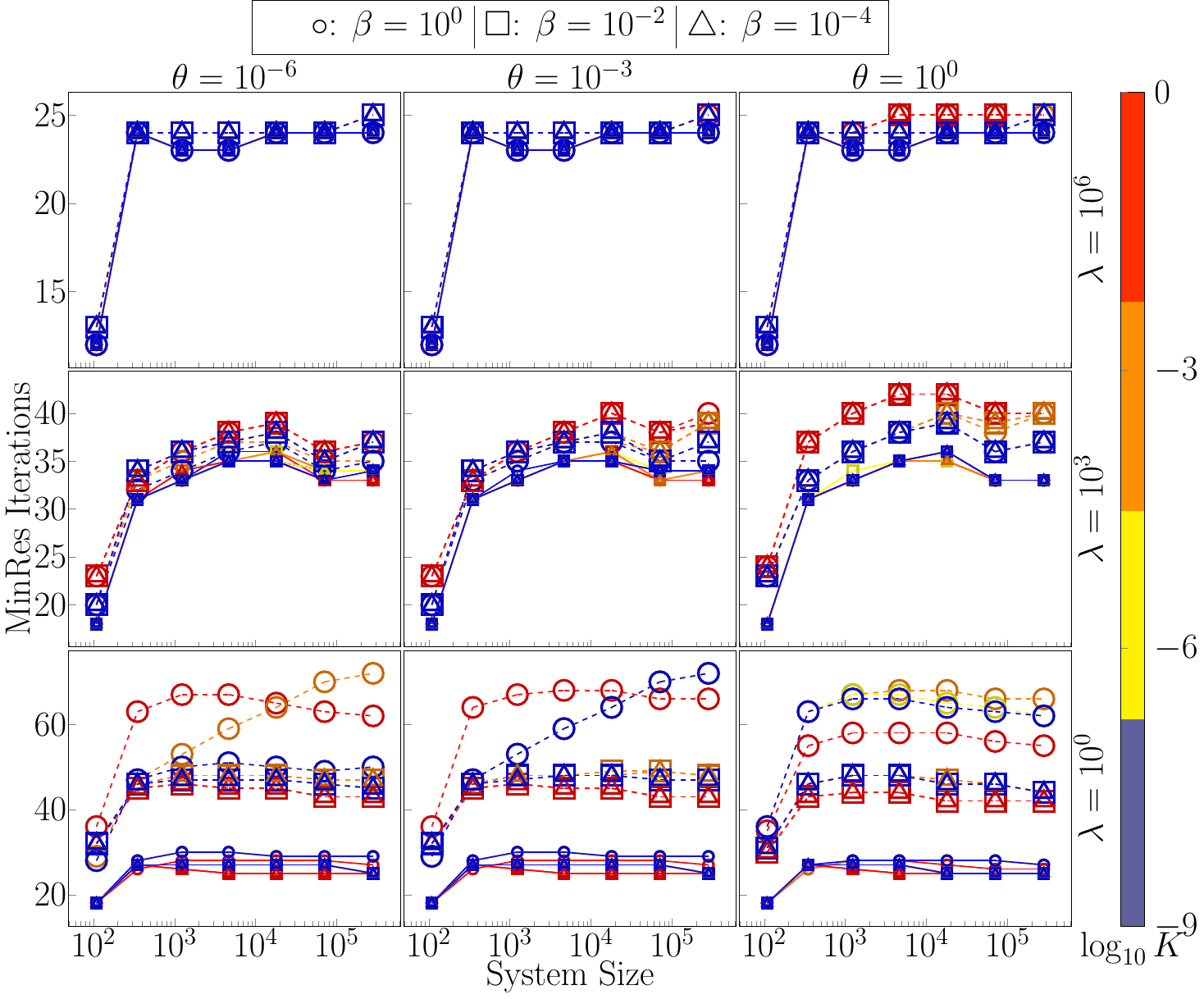
    }
    \includegraphics[width=0.495\textwidth]{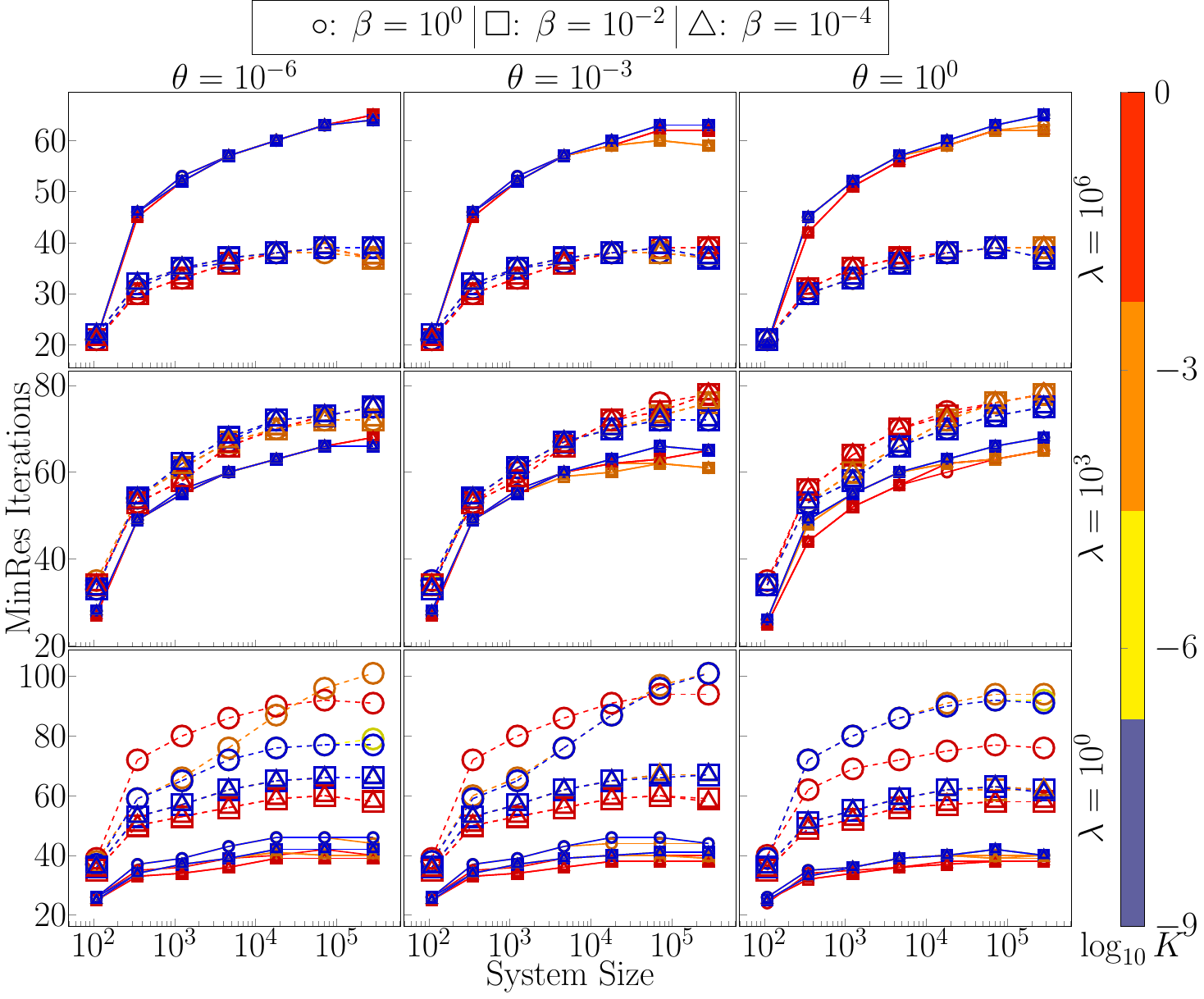
    }
    \vspace{-15pt}
    \caption{
      Performance of preconditioner $\mathcal{B}^{-1}_1$ in \eqref{eq:precond3x3} (dashed lines) and block diagonal preconditioner $\mathcal{B}^{-1}_2$ in  \eqref{eq:preconddiag} (solid lines)
      for linear system due to \eqref{TP_Model_static} under varying $K$ (encoded by line color), $\beta$ (encoded
      by marks), $\theta$ (varies in subplots column-wise) and $\lambda$ (varies in subplots row-wise).
      Two realizations of the preconditioners are compared: (left) exact preconditioners, (right)
      approximation in terms of AMG.
    }
    \label{fig:robust}
  \end{figure}

  Finally, we consider sensitivity to variations in $a_0$, $c_0$, holding the remaining parameters in \eqref{TP_Model} fixed at unit value. Our results
  are summarized in \Cref{tab:example2}, where we report the iterations corresponding
  to mesh refinement levels $l$ with the mesh size $h=h_0/2^l$. We observe
  that the iterations are stable under varying $h$, $a_0$, and $c_0$.

  \begin{table}
    \centering
    \tiny{
      \begin{tabular}{c|cccc||cccc}
        \hline
        & \multicolumn{4}{c||}{LU} & \multicolumn{4}{c}{AMG}\\
        \hline
      \backslashbox{$(a_0, c_0)$}{$l$} & 3 & 4 & 5 & 6 & 3 & 4 & 5 & 6\\ 
      \hline
  $(10^{1}, 10^{1})$ & (25, 36) & (25, 36) & (25, 35) & (24, 35)  & (35, 47) & (36, 48) & (37, 49) & (37, 49)\\
  $(10^{1}, 10^{3})$ & (24, 32) & (24, 32) & (24, 32) & (23, 32)  & (35, 43) & (36, 45) & (36, 46) & (36, 46)\\
  $(10^{1}, 10^{9})$ & (12, 16) & (12, 16) & (12, 16) & (12, 16)  & (17, 20) & (18, 22) & (18, 22) & (18, 22)\\
  $(10^{3}, 10^{1})$ & (24, 32) & (24, 32) & (24, 32) & (23, 32)  & (35, 43) & (36, 45) & (36, 46) & (36, 46)\\
  $(10^{3}, 10^{3})$ & (23, 26) & (23, 26) & (23, 26) & (22, 26)  & (33, 36) & (35, 37) & (35, 38) & (35, 38)\\
  $(10^{3}, 10^{9})$ & (12, 13) & (12, 13) & (12, 13) & (12, 13)  & (17, 17) & (18, 19) & (18, 20) & (18, 20)\\
  $(10^{9}, 10^{1})$ & (12, 16) & (12, 16) & (12, 16) & (12, 16)  & (17, 20) & (18, 22) & (18, 22) & (18, 22)\\
  $(10^{9}, 10^{3})$ & (12, 13) & (12, 13) & (12, 13) & (12, 13)  & (17, 17) & (18, 19) & (18, 20) & (18, 20)\\
  $(10^{9}, 10^{9})$ & (10, 10) & (10, 10) & (10, 10) & (10, 11)  & (16, 17) & (16, 18) & (18, 18) & (18, 18)\\
      \hline
      \end{tabular}
    }
\caption{
      Performance of preconditioners \eqref{eq:precond3x3} and \eqref{eq:preconddiag} 
      for linear system due to \eqref{TP_Model_static} under varying $a_0$, $c_0$.
      The number of MinRes iterations is shown in the brackets, where the first item is obtained
      using the preconditioner \eqref{eq:precond3x3}. Realizations of the preconditioners in terms
      of LU and AMG are compared. 
    }
    \label{tab:example2}
  \end{table}

Notably, the limiting case where all coefficients $a_0$, $b_0$, $c_0$ are all equal to zero holds particular practical significance in fields such as geophysics. Our proposed preconditioners demonstrate robust performance in this challenging scenario, specifically the first preconditioner \eqref{eq:precond3x3}. For the second preconditioner \eqref{eq:preconddiag}, the theoretical guarantees cannot be formally extended to the case $a_0=b_0=c_0$ due to the existence of assumptions \eqref{assumptions:Cp CT }.
However, the numerical implementation remains functionally effective, particularly as $\lambda\to+\infty$. 
In this limit, the range of $C_p, C_T$ extends to include non-negative values, implying that $a_0$, $b_0$, and $c_0$ may take zero values. 
In Table \ref{tab:a0=b0=c0=0}, we present the results in the situation $a_0=b_0=c_0=0$ with varying $\lambda$, still holding the remaining parameters fixed at unit value.
  \begin{table} 
    \centering
    \tiny{
      \begin{tabular}{c|cccc||cccc}
        \hline
        & \multicolumn{4}{c||}{LU} & \multicolumn{4}{c}{AMG}\\
        \hline
      \backslashbox{$\lambda$}{$l$} & 3 & 4 & 5 & 6 & 3 & 4 & 5 & 6\\ 
      \hline
  $10^6$ & (23,24) & (24,24) & (24,26)& (24,26) & (51,58) & (56,64) & (59,68)& (62,70)\\
  $10^3$ & (34,38) & (36,40) & (36,40) & (34,38) & (52,59) & (57,65)& (60,70)& (63,73)\\
  $10^0$ & (26,74) & (26,74) & (26,74) & (26,71)&(33,88) & (35,94) & (36,98)&(36,102)\\
      \hline
      \end{tabular}
    }
\caption{
      Performance of preconditioners \eqref{eq:precond3x3} and \eqref{eq:preconddiag} 
      for linear system due to \eqref{TP_Model_static} under varying $\lambda$.
      The number of MinRes iterations is shown in the brackets, where the first item is obtained
      using the preconditioner \eqref{eq:precond3x3}. Realizations of the preconditioners in terms
      of LU and AMG are compared. 
    }                                     
    \label{tab:a0=b0=c0=0}
  \end{table}

\end{example}

In the final example we discuss the performance of the proposed preconditioners in the case of spatially-varying model parameters.

\begin{example}[Heterogeneous coefficients]\label{ex:heterog}
  We consider \eqref{general TP_Model} in two configurations with discontinuous material
  coefficients inspired by \cite{cai2015overlapping, cai2016hybrid}. The coefficients shall be defined as
  piecewise constant for the $4\times 4$ partition of the domain $\Omega=(0, 1)^2$,
  cf. \Cref{fig:heterog}. In both cases we set $\alpha=3$, $c_0=0.3$, $\beta=2$, $a_0=4$, $b_0=0.1$.
  In the first test case, referred to as "central-jump", we let the Poisson ratio, $\nu$, take the value $\nu=\nu_i=0.4, 0.49, 0.499$ in the middle $2\times 2$ (black) subdomain while $\nu=0.3$
  elsewhere (white subdomain). Furthermore we set $K=1$ and $\theta=1$.

  In the second "checkerboard" test case, the values of $\nu$, $K$, and $\theta$
  vary in the subdomains according to the following lookup tables
  {\footnotesize{  
      \[
      \begin{aligned}
\nu = \left[\begin{array}{cccc}
0.49999 & 0.37 & 0.499 & 0.41 \\
0.3 & 0.49999 & 0.33 & 0.4999 \\
0.49999 & 0.29 & 0.499 & 0.3 \\
0.2 & 0.4999 & 0.31 & 0.499
  \end{array}\right]&,\
K=
\left[\begin{array}{cccc}
100 & 0.1 & 1  & 0.001 \\
0.01 & 1 & 10000 & 0.1 \\
0.09 & 0.29 & 1000 & 0.3 \\
0.001 & 0.9 & 0.1 & 0.01
\end{array}\right],\\
&\theta = 
\left[\begin{array}{cccc}
10 & 0.01 & 0.1  & 10 \\
0.1 & 10000 & 10 & 0.1 \\
0.9 & 0.09  & 10 & 0.3 \\
0.1 & 0.09 & 10  & 200
  \end{array}\right].
\end{aligned}%
\]}}
We note that the black subdomains in \Cref{fig:heterog} are associated with almost incompressible materials,
while the white subdomains are associated with compressible materials. Setting $E=6000$ in each test case the Lam{\'e} coefficients
are computed as
\begin{equation}
\lambda=\frac{E \nu }{(1+ \nu)(1-2 \nu)}, \quad \mu=\frac{E}{2(1+\nu)}.
\end{equation}

With both configurations of material coefficients, we set up the boundary conditions
and forcing terms similar to \Cref{ex:precon}. That is, on the entire $\partial\Omega$
the displacement, pressure and temperature are prescribed based on \eqref{Example1} (with $t=0$).
Furthermore, we compute the forcing terms from \eqref{Example1} using unit values for the spatially
varying coefficients, i.e., $\mu$, $\lambda$, $\theta$, and $K$, while taking their actual
values for the remaining constant coefficients. Fixing $\Delta t=10^{-2}$, we then solve the
linear system due to \eqref{TP_Model_static} and discretization \eqref{discreteFESpace} by a preconditioned MinRes solver starting
from 0 initial guess and convergence tolerance $10^{-12}$. As can be seen in \Cref{tab:heterog}, the preconditioners can handle heterogeneous coefficients when computed exactly as well as
when approximated by AMG. Similar to the constant coefficient tests in \Cref{ex:precon}, fewer iterations
are required with \eqref{eq:precond3x3}. Solutions of the test problems are plotted in \Cref{fig:heterog}.

\begin{figure}
  \setlength{\tabcolsep}{0pt}
  \begin{tabular}{c|ccc}
  \includegraphics[height=2.9cm]{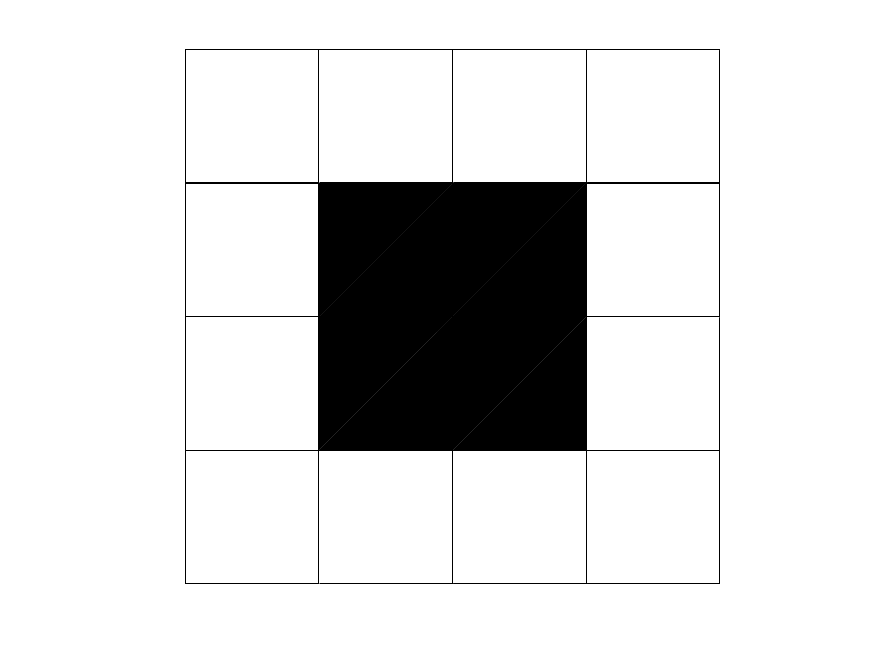} &
  \includegraphics[height=3.1cm]{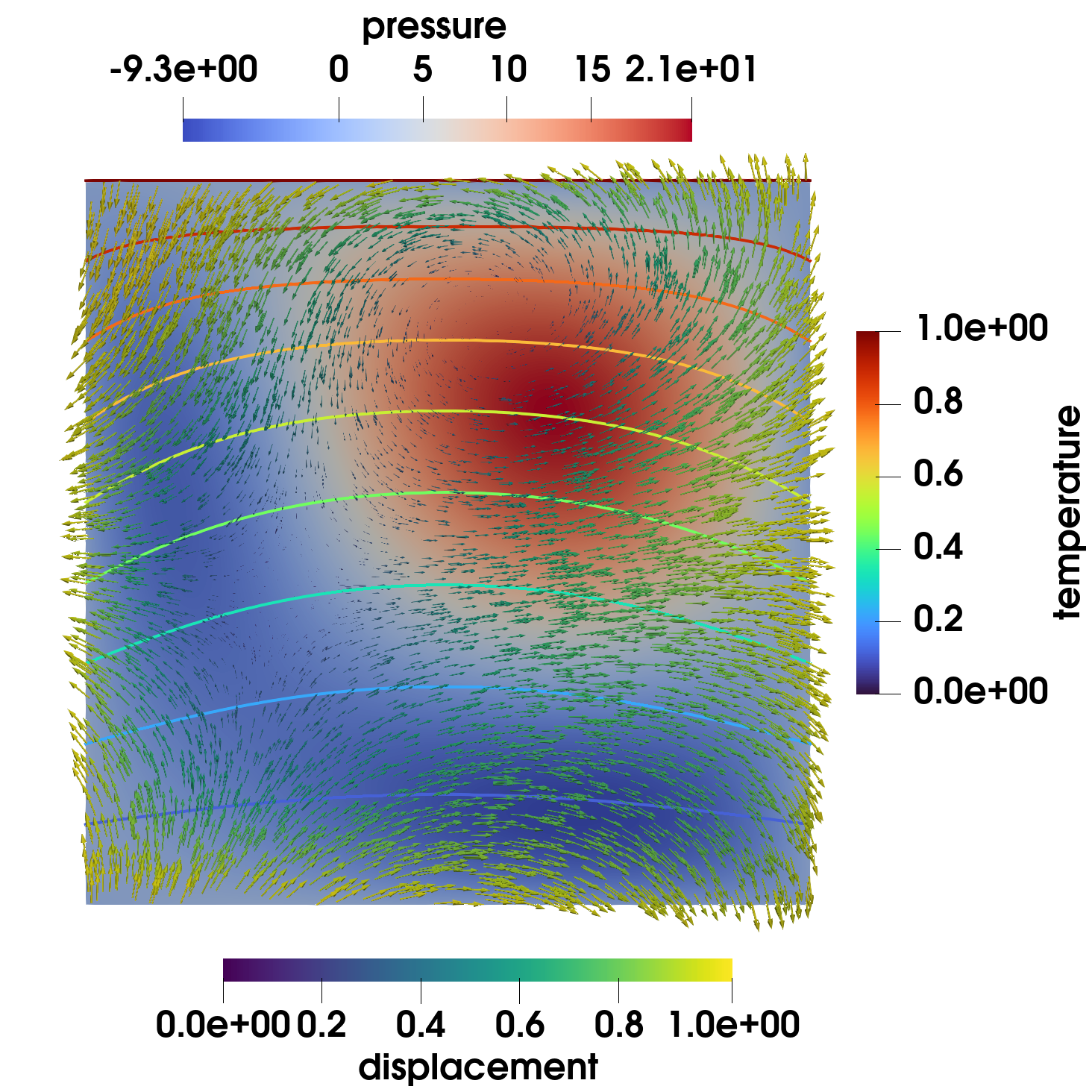} &
  \includegraphics[height=3.1cm]{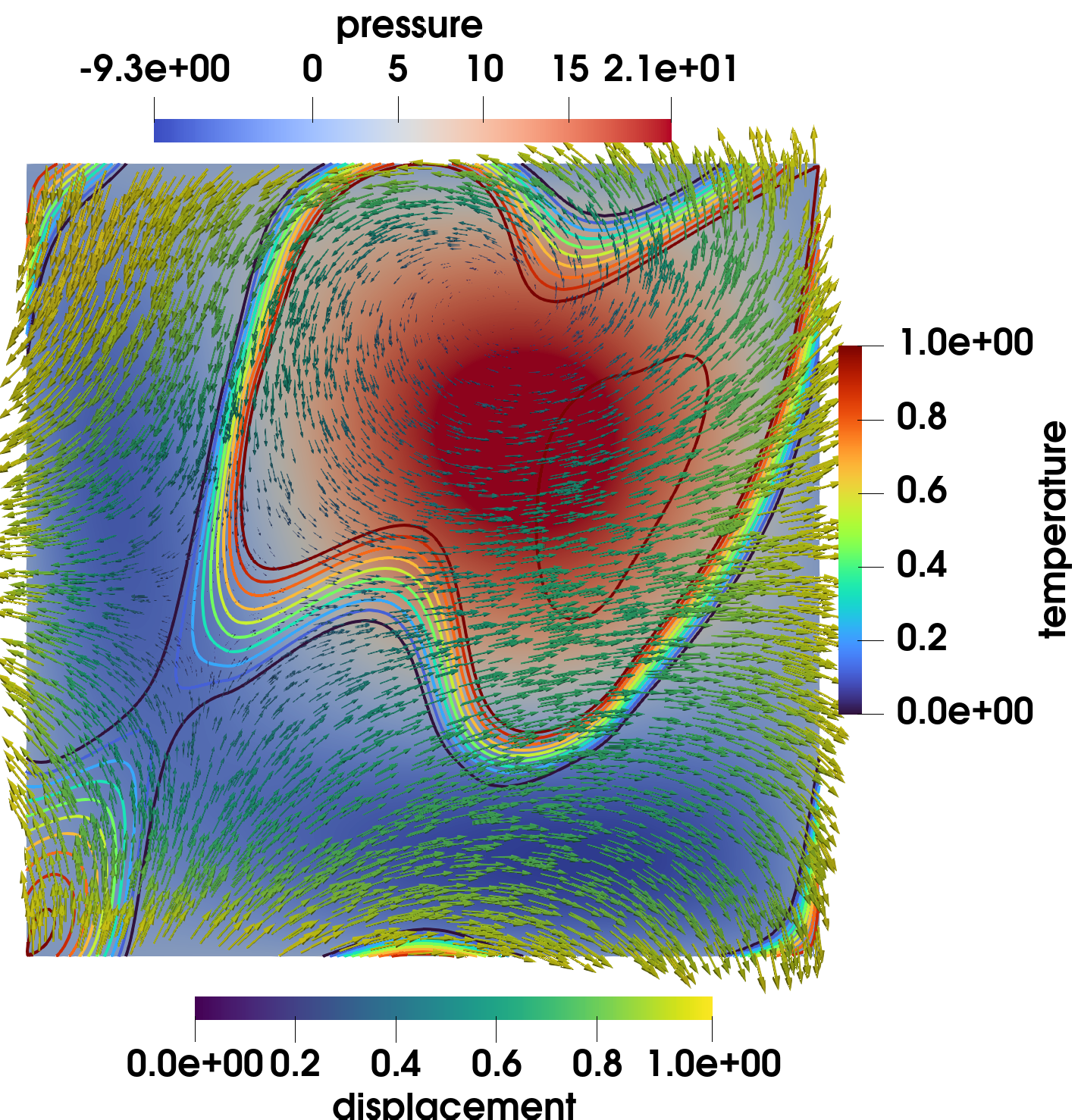} & 
  \includegraphics[height=3.1cm]{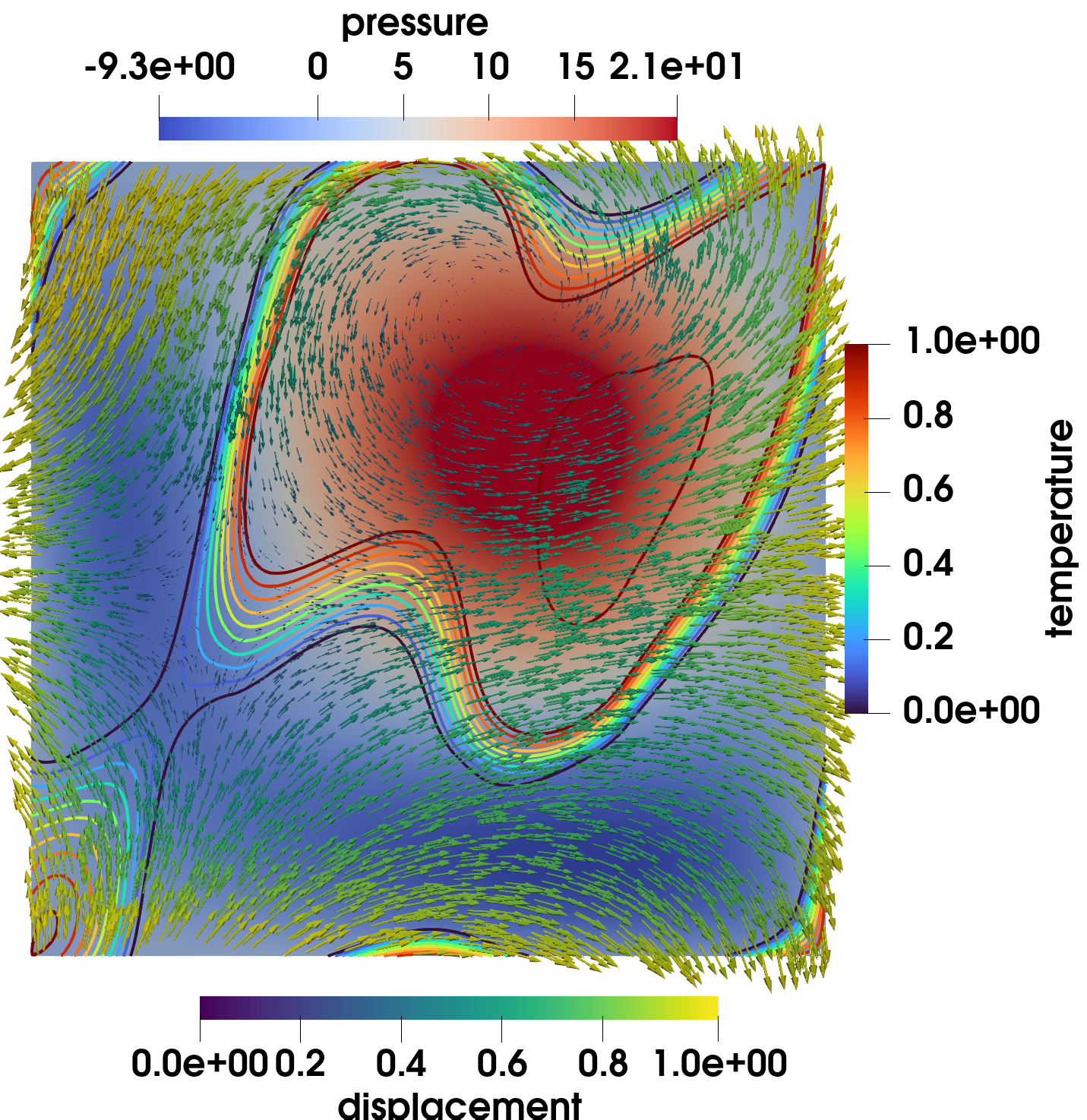}\\
  \includegraphics[height=2.9cm]{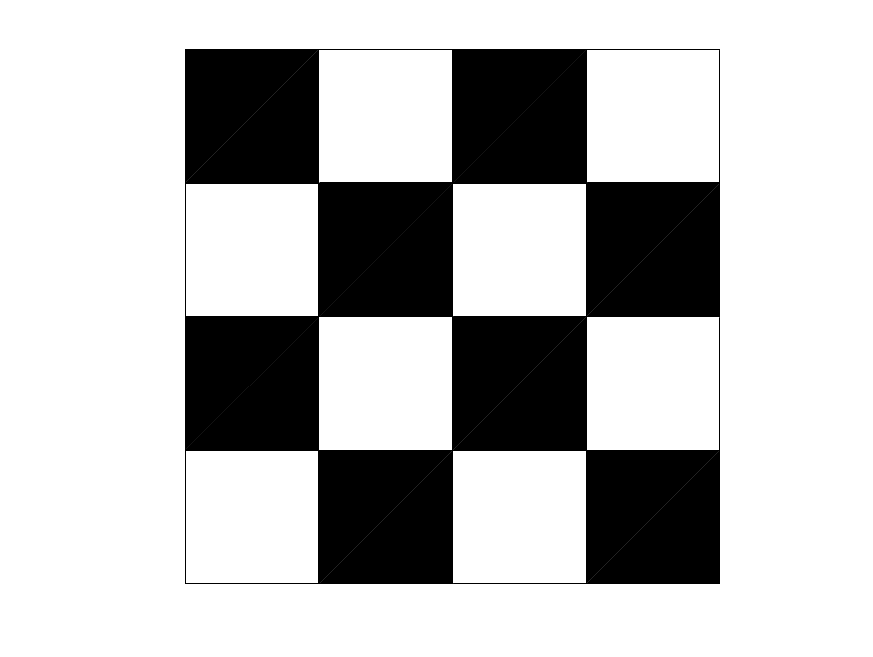} & 
  \includegraphics[height=3.1cm]{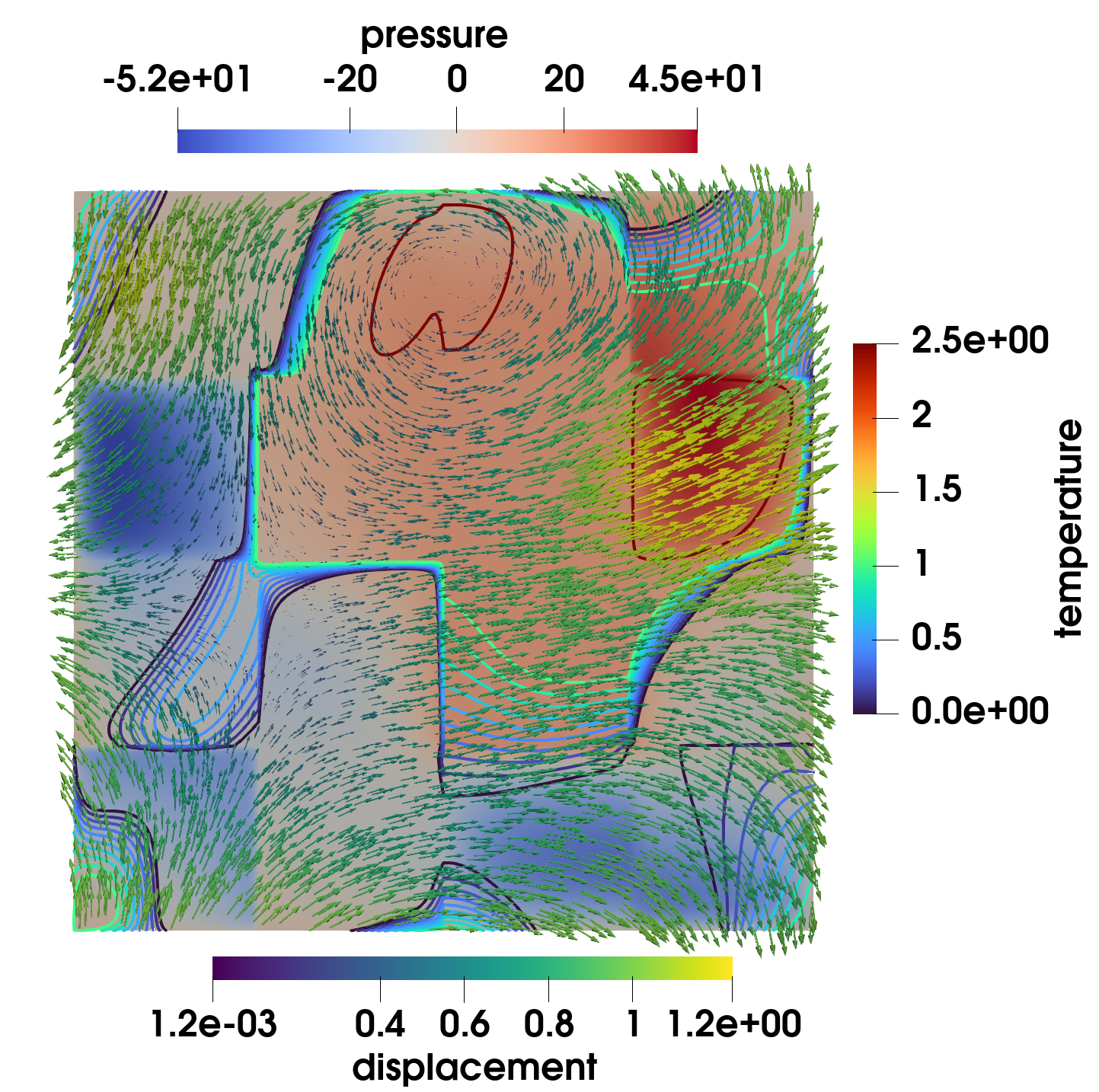} \\
  \end{tabular}
  \caption{
    Configurations of heterogeneous material properties. The coefficients
    of the central-jump and checkerboard test cases are 
    defined as piecewise constant with respect to the partitions shown in the first column.
    Solution of \eqref{TP_Model_static} with $\Delta t=10^{-2}$ and heterogeneous
    material coefficients in central-jump with $\nu_i=0.4, 0.49, 0.499$ (top row) and checkerboard (bottom row)  configurations.
    Displacement visualized by glyphs, pressure by heatmap and temperature by contour lines.
  }
  \label{fig:heterog}
\end{figure}

\begin{table}
  \centering
  \scriptsize{
    \setlength{\tabcolsep}{2pt}
    \begin{tabular}{c|c|c||lllllll}
      \hline
      case & preconditioner & \backslashbox{realization}{$l$} & 2 & 3 & 4 & 5 & 6 & 7 & 8\\
      \hline
      \multirow{4}{*}{\minitab[c]{central-jump\\$\nu_i$=0.4}} & \eqref{eq:precond3x3} & LU  & 23 & 23 &22 &21 &21 &21 &20\\
      & \eqref{eq:precond3x3} & AMG & 25 & 28 &29 &29 &29 &29 &29\\
      & \eqref{eq:preconddiag} & LU & 25 & 26 &24 &24 &24 &22 &22\\
      & \eqref{eq:preconddiag} & AMG& 29 & 30 &32 &32 &32 &32 &31\\
      \hline
      \multirow{4}{*}{\minitab[c]{central-jump\\$\nu_i$=0.49}} & \eqref{eq:precond3x3}  & LU  &  24        &   26        &   26         &    25         &    24         &     24         & 22           \\
   & \eqref{eq:precond3x3}  & AMG &  27        &   31        &   32         &    34         &    33         &     33         & 33           \\
   & \eqref{eq:preconddiag} & LU &  27        &   29        &   27         &    27         &    27         &     25         & 25           \\
       & \eqref{eq:preconddiag} & AMG &  30        &   34        &   35         &    36         &    36         &     36         & 36           \\
       \hline
\multirow{4}{*}{\minitab[c]{central-jump\\$\nu_i$=0.499}} & \eqref{eq:precond3x3}  & LU  &  24        &   26        &   26         &    26         &    24         &     24         & 24           \\
 & \eqref{eq:precond3x3}  & AMG &  27        &   31        &   32         &    34         &    35         &     35         & 35           \\
 & \eqref{eq:preconddiag} & LU  &  27        &   29        &   29         &    27         &    27         &     25         & 25           \\
 & \eqref{eq:preconddiag} & AMG &  30        &   34        &   36         &    38         &    37         &     37         & 37           \\
      \hline                                           
      \multirow{4}{*}{checkerboard} & \eqref{eq:precond3x3} & LU  & 32 & 42 &42 &42 &42 &42 &40\\
      & \eqref{eq:precond3x3} & AMG & 37 & 47 &50 &53 &52 &53 &52\\
      & \eqref{eq:preconddiag} & LU & 37 & 44 &45 &43 &43 &43 &43\\
      & \eqref{eq:preconddiag} & AMG& 39 & 48 &53 &54 &55 &54 &54\\
      \hline
    \end{tabular}
  }
  \caption{Number of MinRes iterations for solving test cases with heterogeneous coefficients (see \Cref{fig:heterog})
    on meshes with size $h=2^{-l}$. Discretization and solver setup as in \Cref{ex:precon}.
  }
  \label{tab:heterog}
\end{table}

\end{example}

\begin{remark}[Dense blocks due to $I_0$]\label{rmrk:dense}
    Due to the operator $I_0$, the matrix representation of both inner products $\mathcal{B}_1$ in \eqref{operator B1} and $\mathcal{B}_2$ in \eqref{operator B2} includes a dense block. In
    particular, let $\bm x\in\mathbb{R}^n$, such that $x_i=P_m\phi_i$ for all
    $\phi_i\in Q_h$ and $n=\text{dim } Q_h$ and let $\bm y = (\bm x, \bm 0, \bm 0)$
    be the representation of $(\phi, 0, 0)$ in $Q_h\times W_h\times W_h$.
    Then, the $Q$-block in the inner product \eqref{operator B2} leading to 
    block-diagonal preconditioner
    \eqref{eq:preconddiag}
    and the $Q\times W\times W$-block in \eqref{operator B1}
    yield operators
    \begin{equation}\label{eq:lr}
    \bm A_2 - \bm x \bm x^T,\quad \bm A_1 - \bm  y \bm y^T.
    \end{equation}
Here, the matrices $\bm A_i$ are symmetric and positive definite and the blocks thus have a structure of low-rank perturbed invertible operators.
    For example,
    with \eqref{operator B2} the $Q$ block, which is the discretization of the inner product of $\lambda^{-1/2}L^2\cap L_0^2$, is represented as $\lambda^{-1} \bm M + (\bm M - \bm x\bm x^T)$
    with $\bm M$ the mass matrix of $Q_h$, see also \cite{lee2017parameter}. In turn, $\bm A_2 = (1+\lambda^{-1})\bm M$. 

When computing the action
    of the preconditioners we take advantage of the structure \eqref{eq:lr} and invoke the
    Sherman–Morrison–Woodbury formula \cite[Ch 2.1]{golub2013matrix}
    \[
    (\bm A - \bm z \bm z^T)^{-1}
    =
    \bm A^{-1} + \frac{1}{1-\bm{z}^T (\bm A^{-1}\bm{z})}(\bm{A}^{-1}\bm{z})\bm{z}^T\bm{A}^{-1}.
    \]
    Finally, the action of $\bm A^{-1}$ in the above formula 
    is either computed with LU decomposition or approximated through AMG V-cycle. We note that, as $\bm A^{-1}\bm z$ can be precomputed once in the setup phase, the subsequent application of the preconditioner requires a single evaluation of $\bm A^{-1}$.
  \end{remark}

  \begin{remark}[AMG in \eqref{eq:precond3x3}]\label{rmrk:amg}
    The $3\times 3$ block of preconditioner \eqref{eq:precond3x3} can present a challenge for parameter-robust approximations in terms of multilevel methods.
    In particular, assuming for simplicity mixed boundary conditions on the
    displacement, and letting $\gamma=\alpha\beta-b_0\lambda$, the inner product operator on $Q\times W\times W$ becomes
    \begin{equation}\label{eq:amg3x3}
    \begin{bmatrix}
      I & 0 & 0 \\
      0 & -t_K\text{div}( \nabla) + c_0 I & 0\\
      0 & 0 & -t_\theta\text{div}(\nabla) + a_0 I 
    \end{bmatrix}
    + \frac{1}{\lambda}
    \begin{bmatrix}
      I          & -\alpha I                 & -\beta I \\
      -\alpha I  & \alpha^2 I     & \gamma I\\
      -\beta  I  & \gamma I  & \beta^2I\\
    \end{bmatrix}
    \end{equation}
    and in certain parameter regimes the second term, which represents the
    coupling between $\xi$, $p$ and $T$, can become singular. 
    As an example, setting $b_0=\beta=0$ the vector $(1, \alpha^{-1}, 0)^T$ can be seen to be
    in the kernel of the coupling operator. In this case, the strength of the singular perturbation
    is controlled by $1/\lambda$.

    Algebraic multigrid methods for singularly perturbed elliptic operators have been developed
    e.g. in \cite{lee2007robust, budivsa2024algebraic}. Therein, a crucial ingredient for uniform
    performance with respect to the strength parameter is smoothers which capture the kernel. However,
    this condition is usually not met by pointwise smoothers, and block smoothers are
    required.
    Motivated by the analysis of \cite{lee2007robust, budivsa2024algebraic} we investigate the
    role of smoothers when realizing the Riesz map with respect to \eqref{eq:amg3x3} by AMG.     
    In order to simplify the realization of block smoothers, let us consider a discretization of
    $Q_h$ and $W_h$ in terms of equal order elements, i.e. using $\mathbb{P}_1$ element for $W_h$ in \eqref{discreteFESpace}. 
    We shall then compare two AMG methods using point smoothers, namely
    BoomerAMG \cite{hypre} and smoothed aggregation AMG (SAMG) \cite{samg}, with SAMG
    utilizing block smoothers. The SAMG implementation was provided\footnote{
    PyAMG supports block smoothers automatically if the matrix representation
    of the problem is in the blocked sparse format. This representation of \eqref{eq:amg3x3} can
    be easily obtained with the equal-element discretization of $Q_h\times W_h\times W_h$.
    } by PyAMG \cite{pyamg2023}.

    We evaluate the different approximations in terms of the spectral condition numbers of
    \eqref{eq:amg3x3} preconditioned by the AMG solvers. In \Cref{tab:amg3x3} we
    observe that the performance of AMG with point smoothers is rather affected by parameter
    variations. On the other hand, block smoothers show little sensitivity. However, let us
    remark that parameter robust estimates in \cite{lee2007robust, budivsa2024algebraic} also
    require specialized prolongation operators that preserve the kernel. Here we have
    used standard prolongation instead. 

    \begin{table}
      \centering
      \tiny{
        \begin{tabular}{l|cccc||cccc||cccc}
          \hline
        & \multicolumn{4}{c||}{BoomerAMG} & \multicolumn{4}{c||}{SAMG, point smoother} & \multicolumn{4}{c}{SAMG, block smoother}\\
        \hline
      \backslashbox{$(\theta, K)$}{$l$} & 3 & 4 & 5 & 6 & 3 & 4 & 5 & 6 & 3 & 4 & 5 & 6 \\           
\hline
 ($10^{-6}$, $10^{-9}$) & 33.6 & 33.6 & 33   & 30.7   & 17.1 & 17   & 16.8 & 15.6  & 1   & 1   & 1   & 1      \\      
 ($10^{-6}$, $10^{-6}$) & 33.6 & 33.5 & 32.5 & 29.2   & 17   & 17   & 16.5 & 14.8  & 1   & 1   & 1   & 1      \\     
 ($10^{-6}$, $10^{-3}$) & 18.6 & 10.2 &  9.8 &  6.5   &  9.5 &  5.4 &  3.9 &  2.6  & 1   & 1   & 1   & 1.2   \\  
 ($10^{-6}$, 1)        &   6.6 &  7.3 &  7.5 &  7.5   &  1.5 &  1.4 &  1.5 &  1.8 &  1.2 & 1.4 & 1.4 & 1.6  \\ 
 \hline
 ($10^{-3}$, $10^{-9}$) & 14.5 &  9.2 &  9.5 &  7.2   &  7.5 &  4.5 &  3.1 &  2    & 1   & 1   & 1.1 & 1.2   \\  
 ($10^{-3}$, $10^{-6}$) & 14.5 &  9.2 &  9.5 &  7.2   &  7.5 &  4.5 &  3.1 &  2    & 1   & 1   & 1.1 & 1.2   \\ 
 ($10^{-3}$, $10^{-3}$) & 11.7 &  9.3 & 10.1 &  8.4   &  6.1 &  4.1 &  2.7 &  1.8  & 1   & 1   & 1.1 & 1.2   \\
 ($10^{-3}$, 1.0)      &   6.6 &  7.2 &  7.4 &  8.2   &  1.4 &  1.4 &  1.5 &  1.8 &  1.2 & 1.3 & 1.4 & 1.7   \\
 \hline
 (1, $10^{-9}$)        &   6   &  6.5 &  6.6 &  6.7   &  1.3 &  1.4 &  1.6 &  1.9 &  1.2 & 1.3 & 1.4 & 1.7   \\
 (1, $10^{-6}$)        &   6   &  6.5 &  6.6 &  6.7   &  1.3 &  1.4 &  1.6 &  1.9 &  1.2 & 1.3 & 1.4 & 1.7   \\
 (1, $10^{-3}$)        &   6   &  6.5 &  6.6 &  6.9   &  1.3 &  1.4 &  1.6 &  1.9 &  1.2 & 1.3 & 1.4 & 1.8   \\
 (1, 1)               &    5.2 &  5.5 &  5.6 &  5.6 &  1.8 &  4.1 & 15.5 & 67.2  & 1.2 & 1.4 & 1.5 & 1.8  \\
\hline
\end{tabular}
      }
      \caption{
        Spectral condition numbers of the Riesz map with respect to the inner product
        on $Q\times Q\times W$ induced by \eqref{operator B1} when using different algebraic multigrid methods 
        for preconditioning. A single V-cycle is always applied.
        Material parameters are as in \Cref{fig:robust}.
      }
      \label{tab:amg3x3}
    \end{table}
  \end{remark}

\section{Conclusions and outlook}\label{sec:conclusions}
In this study, we investigated a four-field formulation of a linear thermo-poroelastic model, discretized using conforming finite elements. To address the challenges associated with parameter variations and discretization, we developed two robust preconditioners for the resulting linear system. Both preconditioners were demonstrated to be robust with respect to model parameters and discretization parameters, as confirmed through extensive numerical experiments. These results highlight the theoretical rigor and practical efficiency of the proposed preconditioners in solving linear thermo-poroelasticity problems in a parameter-robust and computationally effective manner.

Future research will focus on extending these methods to nonlinear thermo-poroelastic models. This includes developing efficient iterative algorithms and advanced preconditioners to handle the additional complexities introduced by nonlinearities. Furthermore, investigating the application of these techniques to large-scale and real-world problems and exploring decoupled and adaptive approaches will be crucial in enhancing the computational efficiency and accuracy of the methods in practical scenarios.

\section*{Acknowledgments}
The work of M. Cai is partially supported by the NIH-RCMI award (Grant No. 347U54MD013376), Army Research Office award W911NF-23-1-0004, and the affiliated project award from the Center for Equitable Artificial Intelligence and Machine Learning Systems (CEAMLS) at Morgan State University (project ID 02232301).
The work of J. Li and Z. Li are partially supported by 
the Shenzhen Sci-Tech Fund No. RCJC20200714114556020, 
Guangdong Basic and Applied Research Fund No. 2023B1515250005.
M. Kuchta gratefully acknowledges support from the RCN grant 303362.  
K.-A. Mardal acknowledges funding and support from Stiftelsen
Kristian Gerhard Jebsen via the K. G. Jebsen Centre for Brain Fluid
Research, the Research Council of Norway via \#300305 (SciML) and \#301013 (Alzheimer´s physics), and the European Research Council under grant  101141807 (aCleanBrain).

\bibliographystyle{siamplain}
\bibliography{references}

\begin{thebibliography}{10}

\bibitem{antonietti2023discontinuous}
{\sc P.~F. Antonietti, S.~Bonetti, and M.~Botti}, {\em Discontinuous {G}alerkin approximation of the fully coupled thermo-poroelastic problem}, SIAM Journal on Scientific Computing, 45 (2023), pp.~A621--A645.

\bibitem{axelsson2012stable}
{\sc O.~Axelsson, R.~Blaheta, and P.~Byczanski}, {\em Stable discretization of poroelasticity problems and efficient preconditioners for arising saddle point type matrices}, Computing and Visualization in Science, 15 (2012), pp.~191--207.

\bibitem{ballarin2024projection}
{\sc F.~Ballarin, S.~Lee, and S.-Y. Yi}, {\em Projection-based reduced order modeling of an iterative scheme for linear thermo-poroelasticity}, Results in Applied Mathematics, 21 (2024), p.~100430.

\bibitem{pyamg2023}
{\sc N.~Bell, L.~N. Olson, J.~Schroder, and B.~Southworth}, {\em {PyAMG}: Algebraic multigrid solvers in {P}ython}, Journal of Open Source Software, 8 (2023), p.~5495.

\bibitem{biot1941general}
{\sc M.~A. Biot}, {\em General theory of three-dimensional consolidation}, Journal of applied physics, 12 (1941), pp.~155--164.

\bibitem{biot1955theory}
{\sc M.~A. Biot}, {\em Theory of elasticity and consolidation for a porous anisotropic solid}, Journal of applied physics, 26 (1955), pp.~182--185.

\bibitem{bonetti2024robust}
{\sc S.~Bonetti, M.~Botti, and P.~F. Antonietti}, {\em Robust discontinuous {G}alerkin-based scheme for the fully-coupled nonlinear thermo-hydro-mechanical problem}, IMA Journal of Numerical Analysis,  (2024), p.~drae045.

\bibitem{boon2021robust}
{\sc W.~M. Boon, M.~Kuchta, K.-A. Mardal, and R.~Ruiz-Baier}, {\em Robust preconditioners for perturbed saddle-point problems and conservative discretizations of {B}iot's equations utilizing total pressure}, SIAM Journal on Scientific Computing, 43 (2021), pp.~B961--B983.

\bibitem{braess1996stability}
{\sc D.~Braess}, {\em Stability of saddle point problems with penalty}, RAIRO Mod\'el. Math. Anal. Num\'er., 30 (1996), pp.~731--742.

\bibitem{brenner2004korn}
{\sc S.~C. Brenner}, {\em Korn's inequalities for piecewise {$H^1$} vector fields}, Mathematics of Computation,  (2004), pp.~1067--1087.

\bibitem{brun2020monolithic}
{\sc M.~K. Brun, E.~Ahmed, I.~Berre, J.~M. Nordbotten, and F.~A. Radu}, {\em Monolithic and splitting solution schemes for fully coupled quasi-static thermo-poroelasticity with nonlinear convective transport}, Computers \& Mathematics with Applications, 80 (2020), pp.~1964--1984.

\bibitem{brun2019well}
{\sc M.~K. Brun, E.~Ahmed, J.~M. Nordbotten, and F.~A. Radu}, {\em Well-posedness of the fully coupled quasi-static thermo-poroelastic equations with nonlinear convective transport}, Journal of Mathematical Analysis and Applications, 471 (2019), pp.~239--266.

\bibitem{brun2018upscaling}
{\sc M.~K. Brun, I.~Berre, J.~M. Nordbotten, and F.~A. Radu}, {\em Upscaling of the coupling of hydromechanical and thermal processes in a quasi-static poroelastic medium}, Transport in Porous Media, 124 (2018), pp.~137--158.

\bibitem{budivsa2021block}
{\sc A.~Budi{\v{s}}a and X.~Hu}, {\em Block preconditioners for mixed-dimensional discretization of flow in fractured porous media}, Computational Geosciences, 25 (2021), pp.~671--686.

\bibitem{budivsa2024algebraic}
{\sc A.~Budi{\v{s}}a, X.~Hu, M.~Kuchta, K.-A. Mardal, and L.~Zikatanov}, {\em Algebraic multigrid methods for metric-perturbed coupled problems}, SIAM Journal on Scientific Computing, 46 (2024), pp.~A1461--A1486.

\bibitem{cai2016hybrid}
{\sc M.~Cai and L.~F. Pavarino}, {\em Hybrid and multiplicative overlapping {S}chwarz algorithms with standard coarse spaces for mixed linear elasticity and stokes problems}, Communications in Computational Physics, 20 (2016), pp.~989--1015.

\bibitem{cai2015overlapping}
{\sc M.~Cai, L.~F. Pavarino, and O.~B. Widlund}, {\em Overlapping {S}chwarz methods with a standard coarse space for almost incompressible linear elasticity}, SIAM Journal on Scientific Computing, 37 (2015), pp.~A811--A830.

\bibitem{chen2022multiphysics}
{\sc Y.~Chen and Z.~Ge}, {\em Multiphysics finite element method for quasi-static thermo-poroelasticity}, Journal of Scientific Computing, 92 (2022), p.~43.

\bibitem{hypre}
{\sc R.~D. Falgout and U.~M. Yang}, {\em hypre: A library of high performance preconditioners}, in Computational Sciences --- ICCS 2002, P.~M.~A. Sloot, A.~G. Hoekstra, C.~J.~K. Tan, and J.~J. Dongarra, eds., Berlin, Heidelberg, 2002, Springer Berlin Heidelberg, pp.~632--641.

\bibitem{ge2023analysis}
{\sc Z.~Ge and D.~Xu}, {\em Analysis of multiphysics finite element method for quasi-static thermo-poroelasticity with a nonlinear convective transport term}, arXiv preprint arXiv:2310.05084,  (2023).

\bibitem{girault2012finite}
{\sc V.~Girault and P.-A. Raviart}, {\em Finite element methods for Navier-Stokes equations: theory and algorithms}, vol.~5, Springer Science \& Business Media, 2012.

\bibitem{golub2013matrix}
{\sc G.~H. Golub and C.~F. Van~Loan}, {\em Matrix computations}, JHU press, 2013.

\bibitem{haga2012parallel}
{\sc J.~B. Haga, H.~Osnes, and H.~P. Langtangen}, {\em A parallel block preconditioner for large-scale poroelasticity with highly heterogeneous material parameters}, Computational Geosciences, 16 (2012), pp.~723--734.

\bibitem{firedrake}
{\sc D.~A. Ham, P.~H.~J. Kelly, L.~Mitchell, C.~J. Cotter, R.~C. Kirby, et~al.}, {\em Firedrake User Manual}, Imperial College London and University of Oxford and Baylor University and University of Washington, first edition~ed., 5 2023.

\bibitem{hong2017parameter}
{\sc Q.~Hong and J.~Kraus}, {\em Parameter-robust stability of classical three-field formulation of {B}iot's consolidation model}, ETNA - Electronic Transactions on Numerical Analysis, 48 (2017).

\bibitem{hong2021new}
{\sc Q.~Hong, J.~K. Kraus, M.~Lymbery, and F.~Philo}, {\em A new practical framework for the stability analysis of perturbed saddle-point problems and applications}, Math. Comput., 92 (2023), pp.~607--634.

\bibitem{kim2018unconditionally}
{\sc J.~Kim}, {\em Unconditionally stable sequential schemes for all-way coupled thermoporomechanics: Undrained-adiabatic and extended fixed-stress splits}, Computer Methods in Applied Mechanics and Engineering, 341 (2018), pp.~93--112.

\bibitem{kolesov2014splitting}
{\sc A.~E. Kolesov, P.~N. Vabishchevich, and M.~V. Vasilyeva}, {\em Splitting schemes for poroelasticity and thermoelasticity problems}, Computers \& Mathematics with Applications, 67 (2014), pp.~2185--2198.

\bibitem{lee2017parameter}
{\sc J.~J. Lee, K.-A. Mardal, and R.~Winther}, {\em Parameter-robust discretization and preconditioning of {B}iot's consolidation model}, SIAM Journal on Scientific Computing, 39 (2017), pp.~A1--A24.

\bibitem{lee2007robust}
{\sc Y.-J. Lee, J.~Wu, J.~Xu, and L.~Zikatanov}, {\em Robust subspace correction methods for nearly singular systems}, Mathematical Models and Methods in Applied Sciences, 17 (2007), pp.~1937--1963.

\bibitem{mardal2011preconditioning}
{\sc K.-A. Mardal and R.~Winther}, {\em Preconditioning discretizations of systems of partial differential equations}, Numerical Linear Algebra with Applications, 18 (2011), pp.~1--40.

\bibitem{oyarzua2016locking}
{\sc R.~Oyarz{\'u}a and R.~Ruiz-Baier}, {\em Locking-free finite element methods for poroelasticity}, SIAM Journal on Numerical Analysis, 54 (2016), pp.~2951--2973.

\bibitem{piersanti2021parameter}
{\sc E.~Piersanti, J.~J. Lee, T.~Thompson, K.-A. Mardal, and M.~E. Rognes}, {\em Parameter robust preconditioning by congruence for multiple-network poroelasticity}, SIAM Journal on Scientific Computing, 43 (2021), pp.~B984--B1007.

\bibitem{selvadurai2016thermo}
{\sc A.~P. Selvadurai and A.~P. Suvorov}, {\em Thermo-poroelasticity and geomechanics}, Cambridge University Press, 2016.

\bibitem{terzaghi1943theoretical}
{\sc K.~Terzaghi}, {\em Theoretical soil mechanics}, 1943.

\bibitem{van2019thermoporoelasticity}
{\sc C.~J. van Duijn, A.~Mikeli{\'c}, M.~F. Wheeler, and T.~Wick}, {\em Thermoporoelasticity via homogenization: modeling and formal two-scale expansions}, International Journal of Engineering Science, 138 (2019), pp.~1--25.

\bibitem{van2020mathematical}
{\sc C.~J. van Duijn, A.~Mikeli{\'c}, and T.~Wick}, {\em Mathematical theory and simulations of thermoporoelasticity}, Computer Methods in Applied Mechanics and Engineering, 366 (2020), p.~113048.

\bibitem{samg}
{\sc P.~Vanek, J.~Mandel, and M.~Brezina}, {\em Algebraic multigrid by smoothed aggregation for second and fourth order elliptic problems}, Computing, 56 (1996), pp.~179--196.

\bibitem{zhang2022galerkin}
{\sc J.~Zhang and H.~Rui}, {\em {G}alerkin method for the fully coupled quasi-static thermo-poroelastic problem}, Computers \& Mathematics with Applications, 118 (2022), pp.~95--109.

\bibitem{zhang2024coupling}
{\sc J.~Zhang and H.~Rui}, {\em A coupling of {G}alerkin and mixed finite element methods for the quasi-static thermo-poroelasticity with nonlinear convective transport}, Journal of Computational and Applied Mathematics, 441 (2024), p.~115672.

\end{thebibliography}

\newpage
\appendix
\section*{Supplementary results}\label{sec:extras}
For the sake of clarity, we summarize the results of the sensitivity study 
in \Cref{fig:robust} in \Cref{tab:robust_lu,tab:robust_amg} showing for each
of the experimental setups (given by material parameters, choice of the preconditioner
and its realization), the number of MinRes iterations required for convergence when the system is discretized on the coarsest and the finest mesh.

\begin{table}
  \centering
  {
    \tiny{
      \renewcommand{\arraystretch}{1.1}
      \setlength{\tabcolsep}{1.1pt}
      \begin{tabular}{l|cc||l|cc||l|cc}
        \hline
        \diagbox{$\beta,\lambda,\theta, K$}{$l$} & 1 & 7 &         \diagbox{$\beta,\lambda,\theta, K$}{$l$} & 1 & 7 &         \diagbox{$\beta,\lambda,\theta, K$}{$l$} & 1 & 7 \\
        \hline
$10^{-4}$,$1$,$10^{-6}$,$10^{-9}$ & (18, 32) & (25, 45) &     $10^{-2}$,$1$,$10^{-6}$,$10^{-9}$ & (18, 32) & (25, 45) &                            $1$,$1$,$10^{-6}$,$10^{-9}$ & (18, 28) & (29, 50) \\       
$10^{-4}$,$1$,$10^{-6}$,$10^{-6}$ & (18, 32) & (25, 45) &     $10^{-2}$,$1$,$10^{-6}$,$10^{-6}$ & (18, 32) & (25, 45) &                            $1$,$1$,$10^{-6}$,$10^{-6}$ & (18, 28) & (29, 50) \\       
$10^{-4}$,$1$,$10^{-6}$,$10^{-3}$ & (18, 32) & (25, 47) &     $10^{-2}$,$1$,$10^{-6}$,$10^{-3}$ & (18, 32) & (25, 47) &                            $1$,$1$,$10^{-6}$,$10^{-3}$ & (18, 29) & (29, 72) \\       
$10^{-4}$,$1$,$10^{-6}$,$1$ & (18, 32) & (25, 43) &           $10^{-2}$,$1$,$10^{-6}$,$1$ & (18, 32) & (25, 43) &                                  $1$,$1$,$10^{-6}$,$1$ & (18, 36) & (27, 62) \\             
$10^{-4}$,$1$,$10^{-3}$,$10^{-9}$ & (18, 32) & (25, 47) &     $10^{-2}$,$1$,$10^{-3}$,$10^{-9}$ & (18, 32) & (25, 47) &                            $1$,$1$,$10^{-3}$,$10^{-9}$ & (18, 29) & (29, 72) \\       
$10^{-4}$,$1$,$10^{-3}$,$10^{-6}$ & (18, 32) & (25, 47) &     $10^{-2}$,$1$,$10^{-3}$,$10^{-6}$ & (18, 32) & (25, 47) &                            $1$,$1$,$10^{-3}$,$10^{-6}$ & (18, 29) & (29, 72) \\       
$10^{-4}$,$1$,$10^{-3}$,$10^{-3}$ & (18, 32) & (25, 48) &     $10^{-2}$,$1$,$10^{-3}$,$10^{-3}$ & (18, 32) & (25, 48) &                            $1$,$1$,$10^{-3}$,$10^{-3}$ & (18, 29) & (29, 72) \\       
$10^{-4}$,$1$,$10^{-3}$,$1$ & (18, 32) & (25, 43) &           $10^{-2}$,$1$,$10^{-3}$,$1$ & (18, 32) & (25, 43) &                                  $1$,$1$,$10^{-3}$,$1$ & (18, 36) & (27, 66) \\             
$10^{-4}$,$1$,$1$,$10^{-9}$ & (18, 31) & (25, 44) &           $10^{-2}$,$1$,$1$,$10^{-9}$ & (18, 31) & (25, 44) &                                  $1$,$1$,$1$,$10^{-9}$ & (18, 36) & (27, 62) \\             
$10^{-4}$,$1$,$1$,$10^{-6}$ & (18, 31) & (25, 44) &           $10^{-2}$,$1$,$1$,$10^{-6}$ & (18, 31) & (25, 44) &                                  $1$,$1$,$1$,$10^{-6}$ & (18, 36) & (27, 62) \\             
$10^{-4}$,$1$,$1$,$10^{-3}$ & (18, 31) & (25, 44) &           $10^{-2}$,$1$,$1$,$10^{-3}$ & (18, 31) & (25, 44) &                                  $1$,$1$,$1$,$10^{-3}$ & (18, 36) & (27, 66) \\             
$10^{-4}$,$1$,$1$,$1$ & (18, 30) & (25, 42) &                 $10^{-2}$,$1$,$1$,$1$ & (18, 30) & (25, 42) &                                        $1$,$1$,$1$,$1$ & (18, 35) & (26, 55) \\                   
$10^{-4}$,$10^{3}$,$10^{-6}$,$10^{-9}$ & (18, 20) & (34, 37) &$10^{-2}$,$10^{3}$,$10^{-6}$,$10^{-9}$ & (18, 20) & (34, 37) &                       $1$,$10^{3}$,$10^{-6}$,$10^{-9}$ & (18, 20) & (34, 35) \\  
$10^{-4}$,$10^{3}$,$10^{-6}$,$10^{-6}$ & (18, 20) & (34, 37) &$10^{-2}$,$10^{3}$,$10^{-6}$,$10^{-6}$ & (18, 20) & (34, 37) &                       $1$,$10^{3}$,$10^{-6}$,$10^{-6}$ & (18, 20) & (34, 35) \\  
$10^{-4}$,$10^{3}$,$10^{-6}$,$10^{-3}$ & (18, 20) & (34, 37) &$10^{-2}$,$10^{3}$,$10^{-6}$,$10^{-3}$ & (18, 20) & (34, 37) &                       $1$,$10^{3}$,$10^{-6}$,$10^{-3}$ & (18, 20) & (34, 35) \\  
$10^{-4}$,$10^{3}$,$10^{-6}$,$1$ & (18, 23) & (33, 37) &      $10^{-2}$,$10^{3}$,$10^{-6}$,$1$ & (18, 23) & (33, 37) &                             $1$,$10^{3}$,$10^{-6}$,$1$ & (18, 23) & (33, 37) \\        
$10^{-4}$,$10^{3}$,$10^{-3}$,$10^{-9}$ & (18, 20) & (34, 37) &$10^{-2}$,$10^{3}$,$10^{-3}$,$10^{-9}$ & (18, 20) & (34, 37) &                       $1$,$10^{3}$,$10^{-3}$,$10^{-9}$ & (18, 20) & (34, 35) \\  
$10^{-4}$,$10^{3}$,$10^{-3}$,$10^{-6}$ & (18, 20) & (34, 37) &$10^{-2}$,$10^{3}$,$10^{-3}$,$10^{-6}$ & (18, 20) & (34, 37) &                       $1$,$10^{3}$,$10^{-3}$,$10^{-6}$ & (18, 20) & (34, 35) \\  
$10^{-4}$,$10^{3}$,$10^{-3}$,$10^{-3}$ & (18, 20) & (34, 39) &$10^{-2}$,$10^{3}$,$10^{-3}$,$10^{-3}$ & (18, 20) & (34, 39) &                       $1$,$10^{3}$,$10^{-3}$,$10^{-3}$ & (18, 20) & (34, 39) \\  
$10^{-4}$,$10^{3}$,$10^{-3}$,$1$ & (18, 23) & (33, 39) &      $10^{-2}$,$10^{3}$,$10^{-3}$,$1$ & (18, 23) & (33, 39) &                             $1$,$10^{3}$,$10^{-3}$,$1$ & (18, 23) & (33, 40) \\        
$10^{-4}$,$10^{3}$,$1$,$10^{-9}$ & (18, 23) & (33, 37) &      $10^{-2}$,$10^{3}$,$1$,$10^{-9}$ & (18, 23) & (33, 37) &                             $1$,$10^{3}$,$1$,$10^{-9}$ & (18, 23) & (33, 37) \\        
$10^{-4}$,$10^{3}$,$1$,$10^{-6}$ & (18, 23) & (33, 37) &      $10^{-2}$,$10^{3}$,$1$,$10^{-6}$ & (18, 23) & (33, 37) &                             $1$,$10^{3}$,$1$,$10^{-6}$ & (18, 23) & (33, 37) \\        
$10^{-4}$,$10^{3}$,$1$,$10^{-3}$ & (18, 23) & (33, 40) &      $10^{-2}$,$10^{3}$,$1$,$10^{-3}$ & (18, 23) & (33, 40) &                             $1$,$10^{3}$,$1$,$10^{-3}$ & (18, 23) & (33, 40) \\        
$10^{-4}$,$10^{3}$,$1$,$1$ & (18, 24) & (33, 40) &            $10^{-2}$,$10^{3}$,$1$,$1$ & (18, 24) & (33, 40) &                                   $1$,$10^{3}$,$1$,$1$ & (18, 24) & (33, 40) \\              
$10^{-4}$,$10^{6}$,$10^{-6}$,$10^{-9}$ & (12, 13) & (24, 25) &$10^{-2}$,$10^{6}$,$10^{-6}$,$10^{-9}$ & (12, 13) & (24, 25) &                       $1$,$10^{6}$,$10^{-6}$,$10^{-9}$ & (12, 12) & (24, 24) \\  
$10^{-4}$,$10^{6}$,$10^{-6}$,$10^{-6}$ & (12, 13) & (24, 25) &$10^{-2}$,$10^{6}$,$10^{-6}$,$10^{-6}$ & (12, 13) & (24, 25) &                       $1$,$10^{6}$,$10^{-6}$,$10^{-6}$ & (12, 12) & (24, 24) \\  
$10^{-4}$,$10^{6}$,$10^{-6}$,$10^{-3}$ & (12, 13) & (24, 25) &$10^{-2}$,$10^{6}$,$10^{-6}$,$10^{-3}$ & (12, 13) & (24, 25) &                       $1$,$10^{6}$,$10^{-6}$,$10^{-3}$ & (12, 12) & (24, 24) \\  
$10^{-4}$,$10^{6}$,$10^{-6}$,$1$ & (12, 13) & (24, 25) &      $10^{-2}$,$10^{6}$,$10^{-6}$,$1$ & (12, 13) & (24, 25) &                             $1$,$10^{6}$,$10^{-6}$,$1$ & (12, 12) & (24, 24) \\        
$10^{-4}$,$10^{6}$,$10^{-3}$,$10^{-9}$ & (12, 13) & (24, 25) &$10^{-2}$,$10^{6}$,$10^{-3}$,$10^{-9}$ & (12, 13) & (24, 25) &                       $1$,$10^{6}$,$10^{-3}$,$10^{-9}$ & (12, 12) & (24, 24) \\  
$10^{-4}$,$10^{6}$,$10^{-3}$,$10^{-6}$ & (12, 13) & (24, 25) &$10^{-2}$,$10^{6}$,$10^{-3}$,$10^{-6}$ & (12, 13) & (24, 25) &                       $1$,$10^{6}$,$10^{-3}$,$10^{-6}$ & (12, 12) & (24, 24) \\  
$10^{-4}$,$10^{6}$,$10^{-3}$,$10^{-3}$ & (12, 13) & (24, 25) &$10^{-2}$,$10^{6}$,$10^{-3}$,$10^{-3}$ & (12, 13) & (24, 25) &                       $1$,$10^{6}$,$10^{-3}$,$10^{-3}$ & (12, 12) & (24, 24) \\  
$10^{-4}$,$10^{6}$,$10^{-3}$,$1$ & (12, 13) & (24, 25) &      $10^{-2}$,$10^{6}$,$10^{-3}$,$1$ & (12, 13) & (24, 25) &                             $1$,$10^{6}$,$10^{-3}$,$1$ & (12, 12) & (24, 25) \\        
$10^{-4}$,$10^{6}$,$1$,$10^{-9}$ & (12, 13) & (24, 25) &      $10^{-2}$,$10^{6}$,$1$,$10^{-9}$ & (12, 13) & (24, 25) &                             $1$,$10^{6}$,$1$,$10^{-9}$ & (12, 12) & (24, 24) \\        
$10^{-4}$,$10^{6}$,$1$,$10^{-6}$ & (12, 13) & (24, 25) &      $10^{-2}$,$10^{6}$,$1$,$10^{-6}$ & (12, 13) & (24, 25) &                             $1$,$10^{6}$,$1$,$10^{-6}$ & (12, 12) & (24, 24) \\        
$10^{-4}$,$10^{6}$,$1$,$10^{-3}$ & (12, 13) & (24, 25) &      $10^{-2}$,$10^{6}$,$1$,$10^{-3}$ & (12, 13) & (24, 25) &                             $1$,$10^{6}$,$1$,$10^{-3}$ & (12, 12) & (24, 25) \\        
$10^{-4}$,$10^{6}$,$1$,$1$ & (12, 13) & (24, 25) &            $10^{-2}$,$10^{6}$,$1$,$1$ & (12, 13) & (24, 25) &                                   $1$,$10^{6}$,$1$,$1$ & (12, 12) & (24, 25) \\              
        \hline                                                
      \end{tabular}                                                  
    }
  }
  \caption{
    Tabular representation of subset of data from \Cref{fig:robust}.
    Number of MinRes iterations required to solve the linear system due to \eqref{TP_Model_static} under varying $K$,
    $\beta$, $\theta$ and $\lambda$ and discretization with meshes of size $h=1/2^l$ using
    preconditioners $\mathcal{B}^{-1}_1$ in \eqref{eq:precond3x3} (corresponding iteration count is shown as the
    first number in the brackets) and block diagonal preconditioner $\mathcal{B}^{-1}_2$ in  \eqref{eq:preconddiag}. The
    preconditioners are computed exactly by LU.   
  }
  \label{tab:robust_lu}
\end{table}

\begin{table}
  \centering
  {
    \tiny{
      \renewcommand{\arraystretch}{1.1}
      \setlength{\tabcolsep}{1.1pt}
      \begin{tabular}{l|cc||l|cc||l|cc}
        \hline
        \diagbox{$\beta,\lambda,\theta, K$}{$l$} & 1 & 7 &         \diagbox{$\beta,\lambda,\theta, K$}{$l$} & 1 & 7 &         \diagbox{$\beta,\lambda,\theta, K$}{$l$} & 1 & 7 \\
        \hline
$10^{-4}$,$1$,$10^{-6}$,$10^{-9}$ & (25, 36) & (42, 66) &         $10^{-2}$,$1$,$10^{-6}$,$10^{-9}$ & (26, 36) & (42, 66) &          $1$,$1$,$10^{-6}$,$10^{-9}$ & (26, 37) & (46, 77) \\       
$10^{-4}$,$1$,$10^{-6}$,$10^{-6}$ & (25, 36) & (42, 66) &         $10^{-2}$,$1$,$10^{-6}$,$10^{-6}$ & (26, 36) & (42, 66) &          $1$,$1$,$10^{-6}$,$10^{-6}$ & (26, 37) & (46, 79) \\       
$10^{-4}$,$1$,$10^{-6}$,$10^{-3}$ & (25, 36) & (40, 66) &         $10^{-2}$,$1$,$10^{-6}$,$10^{-3}$ & (26, 36) & (40, 66) &          $1$,$1$,$10^{-6}$,$10^{-3}$ & (26, 38) & (44, 101) \\      
$10^{-4}$,$1$,$10^{-6}$,$1$ & (25, 35) & (39, 58) &               $10^{-2}$,$1$,$10^{-6}$,$1$ & (25, 35) & (39, 58) &                $1$,$1$,$10^{-6}$,$1$ & (26, 39) & (40, 91) \\             
$10^{-4}$,$1$,$10^{-3}$,$10^{-9}$ & (25, 36) & (41, 67) &         $10^{-2}$,$1$,$10^{-3}$,$10^{-9}$ & (26, 36) & (41, 67) &          $1$,$1$,$10^{-3}$,$10^{-9}$ & (26, 38) & (44, 101) \\      
$10^{-4}$,$1$,$10^{-3}$,$10^{-6}$ & (25, 36) & (41, 67) &         $10^{-2}$,$1$,$10^{-3}$,$10^{-6}$ & (26, 36) & (41, 67) &          $1$,$1$,$10^{-3}$,$10^{-6}$ & (26, 38) & (44, 101) \\      
$10^{-4}$,$1$,$10^{-3}$,$10^{-3}$ & (25, 36) & (39, 67) &         $10^{-2}$,$1$,$10^{-3}$,$10^{-3}$ & (26, 36) & (39, 67) &          $1$,$1$,$10^{-3}$,$10^{-3}$ & (26, 38) & (44, 101) \\      
$10^{-4}$,$1$,$10^{-3}$,$1$ & (25, 35) & (38, 58) &               $10^{-2}$,$1$,$10^{-3}$,$1$ & (25, 35) & (38, 59) &                $1$,$1$,$10^{-3}$,$1$ & (26, 39) & (40, 94) \\             
$10^{-4}$,$1$,$1$,$10^{-9}$ & (25, 36) & (40, 61) &               $10^{-2}$,$1$,$1$,$10^{-9}$ & (25, 36) & (40, 62) &                $1$,$1$,$1$,$10^{-9}$ & (26, 39) & (40, 91) \\             
$10^{-4}$,$1$,$1$,$10^{-6}$ & (25, 36) & (40, 61) &               $10^{-2}$,$1$,$1$,$10^{-6}$ & (25, 36) & (40, 62) &                $1$,$1$,$1$,$10^{-6}$ & (26, 39) & (40, 92) \\             
$10^{-4}$,$1$,$1$,$10^{-3}$ & (25, 36) & (39, 62) &               $10^{-2}$,$1$,$1$,$10^{-3}$ & (25, 36) & (40, 62) &                $1$,$1$,$1$,$10^{-3}$ & (26, 39) & (40, 94) \\             
$10^{-4}$,$1$,$1$,$1$ & (25, 35) & (38, 58) &                     $10^{-2}$,$1$,$1$,$1$ & (25, 35) & (38, 58) &                      $1$,$1$,$1$,$1$ & (24, 40) & (38, 76) \\                   
$10^{-4}$,$10^{3}$,$10^{-6}$,$10^{-9}$ & (28, 33) & (66, 75) &    $10^{-2}$,$10^{3}$,$10^{-6}$,$10^{-9}$ & (28, 33) & (66, 75) &     $1$,$10^{3}$,$10^{-6}$,$10^{-9}$ & (28, 33) & (66, 75) \\  
$10^{-4}$,$10^{3}$,$10^{-6}$,$10^{-6}$ & (28, 33) & (66, 75) &    $10^{-2}$,$10^{3}$,$10^{-6}$,$10^{-6}$ & (28, 33) & (66, 75) &     $1$,$10^{3}$,$10^{-6}$,$10^{-6}$ & (28, 33) & (66, 75) \\  
$10^{-4}$,$10^{3}$,$10^{-6}$,$10^{-3}$ & (28, 33) & (66, 72) &    $10^{-2}$,$10^{3}$,$10^{-6}$,$10^{-3}$ & (28, 33) & (66, 72) &     $1$,$10^{3}$,$10^{-6}$,$10^{-3}$ & (28, 35) & (66, 72) \\  
$10^{-4}$,$10^{3}$,$10^{-6}$,$1$ & (27, 34) & (68, 75) &          $10^{-2}$,$10^{3}$,$10^{-6}$,$1$ & (27, 34) & (68, 75) &           $1$,$10^{3}$,$10^{-6}$,$1$ & (27, 34) & (68, 75) \\        
$10^{-4}$,$10^{3}$,$10^{-3}$,$10^{-9}$ & (28, 33) & (65, 72) &    $10^{-2}$,$10^{3}$,$10^{-3}$,$10^{-9}$ & (28, 33) & (65, 72) &     $1$,$10^{3}$,$10^{-3}$,$10^{-9}$ & (28, 35) & (65, 72) \\  
$10^{-4}$,$10^{3}$,$10^{-3}$,$10^{-6}$ & (28, 33) & (65, 72) &    $10^{-2}$,$10^{3}$,$10^{-3}$,$10^{-6}$ & (28, 33) & (65, 72) &     $1$,$10^{3}$,$10^{-3}$,$10^{-6}$ & (28, 34) & (65, 72) \\  
$10^{-4}$,$10^{3}$,$10^{-3}$,$10^{-3}$ & (28, 33) & (61, 76) &    $10^{-2}$,$10^{3}$,$10^{-3}$,$10^{-3}$ & (28, 33) & (61, 76) &     $1$,$10^{3}$,$10^{-3}$,$10^{-3}$ & (28, 35) & (61, 76) \\  
$10^{-4}$,$10^{3}$,$10^{-3}$,$1$ & (27, 34) & (65, 78) &          $10^{-2}$,$10^{3}$,$10^{-3}$,$1$ & (27, 34) & (65, 78) &           $1$,$10^{3}$,$10^{-3}$,$1$ & (27, 34) & (65, 78) \\        
$10^{-4}$,$10^{3}$,$1$,$10^{-9}$ & (26, 34) & (68, 75) &          $10^{-2}$,$10^{3}$,$1$,$10^{-9}$ & (26, 34) & (68, 75) &           $1$,$10^{3}$,$1$,$10^{-9}$ & (26, 34) & (68, 75) \\        
$10^{-4}$,$10^{3}$,$1$,$10^{-6}$ & (26, 34) & (68, 75) &          $10^{-2}$,$10^{3}$,$1$,$10^{-6}$ & (26, 34) & (68, 75) &           $1$,$10^{3}$,$1$,$10^{-6}$ & (26, 34) & (68, 75) \\        
$10^{-4}$,$10^{3}$,$1$,$10^{-3}$ & (26, 34) & (65, 78) &          $10^{-2}$,$10^{3}$,$1$,$10^{-3}$ & (26, 34) & (65, 78) &           $1$,$10^{3}$,$1$,$10^{-3}$ & (26, 34) & (65, 78) \\        
$10^{-4}$,$10^{3}$,$1$,$1$ & (25, 34) & (65, 78) &                $10^{-2}$,$10^{3}$,$1$,$1$ & (25, 34) & (65, 78) &                 $1$,$10^{3}$,$1$,$1$ & (25, 35) & (65, 78) \\              
$10^{-4}$,$10^{6}$,$10^{-6}$,$10^{-9}$ & (21, 22) & (64, 39) &    $10^{-2}$,$10^{6}$,$10^{-6}$,$10^{-9}$ & (21, 22) & (64, 39) &     $1$,$10^{6}$,$10^{-6}$,$10^{-9}$ & (21, 21) & (64, 39) \\  
$10^{-4}$,$10^{6}$,$10^{-6}$,$10^{-6}$ & (21, 22) & (64, 39) &    $10^{-2}$,$10^{6}$,$10^{-6}$,$10^{-6}$ & (21, 22) & (64, 39) &     $1$,$10^{6}$,$10^{-6}$,$10^{-6}$ & (21, 21) & (64, 39) \\  
$10^{-4}$,$10^{6}$,$10^{-6}$,$10^{-3}$ & (21, 22) & (64, 37) &    $10^{-2}$,$10^{6}$,$10^{-6}$,$10^{-3}$ & (21, 22) & (64, 37) &     $1$,$10^{6}$,$10^{-6}$,$10^{-3}$ & (21, 21) & (64, 37) \\  
$10^{-4}$,$10^{6}$,$10^{-6}$,$1$ & (21, 21) & (65, 37) &          $10^{-2}$,$10^{6}$,$10^{-6}$,$1$ & (21, 21) & (65, 37) &           $1$,$10^{6}$,$10^{-6}$,$1$ & (21, 21) & (65, 37) \\        
$10^{-4}$,$10^{6}$,$10^{-3}$,$10^{-9}$ & (21, 22) & (63, 37) &    $10^{-2}$,$10^{6}$,$10^{-3}$,$10^{-9}$ & (21, 22) & (63, 37) &     $1$,$10^{6}$,$10^{-3}$,$10^{-9}$ & (21, 21) & (63, 37) \\  
$10^{-4}$,$10^{6}$,$10^{-3}$,$10^{-6}$ & (21, 22) & (63, 37) &    $10^{-2}$,$10^{6}$,$10^{-3}$,$10^{-6}$ & (21, 22) & (63, 37) &     $1$,$10^{6}$,$10^{-3}$,$10^{-6}$ & (21, 21) & (63, 37) \\  
$10^{-4}$,$10^{6}$,$10^{-3}$,$10^{-3}$ & (21, 22) & (59, 37) &    $10^{-2}$,$10^{6}$,$10^{-3}$,$10^{-3}$ & (21, 22) & (59, 37) &     $1$,$10^{6}$,$10^{-3}$,$10^{-3}$ & (21, 21) & (59, 37) \\  
$10^{-4}$,$10^{6}$,$10^{-3}$,$1$ & (21, 21) & (62, 39) &          $10^{-2}$,$10^{6}$,$10^{-3}$,$1$ & (21, 21) & (62, 39) &           $1$,$10^{6}$,$10^{-3}$,$1$ & (21, 21) & (62, 39) \\        
$10^{-4}$,$10^{6}$,$1$,$10^{-9}$ & (20, 21) & (65, 37) &          $10^{-2}$,$10^{6}$,$1$,$10^{-9}$ & (20, 21) & (65, 37) &           $1$,$10^{6}$,$1$,$10^{-9}$ & (20, 21) & (65, 37) \\        
$10^{-4}$,$10^{6}$,$1$,$10^{-6}$ & (20, 21) & (65, 37) &          $10^{-2}$,$10^{6}$,$1$,$10^{-6}$ & (20, 21) & (65, 37) &           $1$,$10^{6}$,$1$,$10^{-6}$ & (20, 21) & (65, 37) \\        
$10^{-4}$,$10^{6}$,$1$,$10^{-3}$ & (20, 21) & (62, 39) &          $10^{-2}$,$10^{6}$,$1$,$10^{-3}$ & (20, 21) & (62, 39) &           $1$,$10^{6}$,$1$,$10^{-3}$ & (20, 21) & (63, 39) \\        
        $10^{-4}$,$10^{6}$,$1$,$1$ & (20, 21) & (62, 39) &                $10^{-2}$,$10^{6}$,$1$,$1$ & (20, 21) & (62, 39) &                 $1$,$10^{6}$,$1$,$1$ & (20, 21) & (62, 39) \\
        \hline
      \end{tabular}                                                  
    }
  }
  \caption{
    Tabular representation of subset of data from \Cref{fig:robust}.
    Number of MinRes iterations required to solve the linear system due to \eqref{TP_Model_static} under varying $K$,
    $\beta$, $\theta$ and $\lambda$ and discretization with meshes of size $h=1/2^l$ using
    preconditioners $\mathcal{B}^{-1}_1$ in \eqref{eq:precond3x3} (corresponding iteration count is shown as the
    first number in the brackets) and block diagonal preconditioner $\mathcal{B}^{-1}_2$ in  \eqref{eq:preconddiag}. The
    preconditioners are approximated by AMG.   
  }
  \label{tab:robust_amg}  
\end{table}

\end{document}